\documentclass[11pt,reqno]{amsart}
\usepackage{fullpage,times,graphicx,amssymb,amsmath,amsfonts,psfrag,xcolor,pdfsync}
\usepackage[colorlinks,citecolor=bbluegray,linkcolor=ddarkbrown,urlcolor=blue,breaklinks]{hyperref}
\newtheorem{theorem}{Theorem}[section]
\newtheorem{proposition}[theorem]{Proposition}
\newtheorem{definition}[theorem]{Definition}
\newtheorem{lemma}[theorem]{Lemma}
\newtheorem{corollary}[theorem]{Corollary}
\renewenvironment{proof}{\textbf{Proof.}}{\QED\bigskip}

\usepackage{algorithmic,algorithm}

\usepackage[square]{natbib}

\definecolor{ddarkbrown}{rgb}{0.5,0.2,0.05} \definecolor{bbluegray}{rgb}{0.05,0,0.5}

\newcommand{\BEAS}{\begin{eqnarray*}}
\newcommand{\EEAS}{\end{eqnarray*}}
\newcommand{\BEA}{\begin{eqnarray}}
\newcommand{\EEA}{\end{eqnarray}}
\newcommand{\BEQ}{\begin{equation}}
\newcommand{\EEQ}{\end{equation}}
\newcommand{\BIT}{\begin{itemize}}
\newcommand{\EIT}{\end{itemize}}
\newcommand{\BNUM}{\begin{enumerate}}
\newcommand{\ENUM}{\end{enumerate}}

\newcommand{\BA}{\begin{array}}
\newcommand{\EA}{\end{array}}

\newcommand{\cf}{{\it cf.}}

\newcommand{\ones}{\mathbf 1}

\newcommand{\reals}{{\mathbb R}}

\newcommand{\symm}{{\mbox{\bf S}}}  

\newcommand{\Tr}{\mathop{\bf Tr}}
\newcommand{\diag}{\mathop{\bf diag}}

\newcommand{\idm}{\mathbf{I}}

\newcommand{\Expect}{\textstyle\mathop{\bf E}}

\newcommand{\Var}{\mathop{\bf var{}}}

\newcommand{\QED}{~~\rule[-1pt]{6pt}{6pt}}

\newcommand{\var}{\mathop{\bf var}}

\newcommand{\vect}[1]{\mathop{\bf vec}(#1)}

\newcommand{\argmax}{\mathop{\rm argmax}}

\usepackage[off]{auto-pst-pdf} 

\title{Convex Relaxations for Permutation Problems}

\author{Fajwel Fogel}
\address{C.M.A.P., \'Ecole Polytechnique,\vskip 0ex
Palaiseau, France.}
\email{fajwel.fogel@cmap.polytechnique.fr}

\author{Rodolphe Jenatton}
\address{C.M.A.P., \'Ecole Polytechnique \& Amazon Labs, Berlin.}
\email{rodolphe.jenatton@gmail.com}

\author{Francis Bach}
\address{INRIA, SIERRA Project-Team \& D.I., \vskip 0ex \'Ecole Normale Sup\'erieure, Paris, France.}
\email{francis.bach@ens.fr}

\author{Alexandre d'Aspremont} 
\address{CNRS \& D.I., UMR 8548, \vskip 0ex
\'Ecole Normale Sup\'erieure, Paris, France.}
\email{aspremon@ens.fr}

\date{\today}

\begin{document}

\begin{abstract} 
Seriation seeks to reconstruct a linear order between variables using unsorted, pairwise similarity information. It has direct applications in archeology and shotgun gene sequencing for example. We write seriation as an optimization problem by proving the equivalence between the seriation and combinatorial 2-SUM problems on similarity matrices (2-SUM is a quadratic minimization problem over permutations). The seriation problem can be solved exactly by a spectral algorithm in the noiseless case and we derive several convex relaxations for 2-SUM to improve the robustness of seriation solutions in noisy settings. These convex relaxations also allow us to impose structural constraints on the solution, hence solve semi-supervised seriation problems. We derive new approximation bounds for some of these relaxations and present numerical experiments on archeological data, Markov chains and DNA assembly from shotgun gene sequencing data.
\end{abstract}

\maketitle

\section{Introduction}

We study optimization problems written over the set of permutations. While the relaxation techniques discussed in what follows are applicable to a much more general setting, most of the paper is centered on the {\em seriation} problem: we are given a similarity matrix between a set of~$n$ variables and assume that the variables can be ordered along a chain, where the similarity between variables decreases with their distance within this chain. The seriation problem seeks to reconstruct this linear ordering based on unsorted, possibly noisy, pairwise similarity information. 

This problem has its roots in archeology~\citep{Robi51} and also has direct applications in e.g. envelope reduction algorithms for sparse linear algebra \citep{Barn95}, in identifying interval graphs for scheduling~\citep{Fulk65}, or in shotgun DNA sequencing where a single strand of genetic material is reconstructed from many cloned shorter reads (i.e.~small, fully sequenced sections of DNA)~\citep{Garr11,Meid98}. With shotgun gene sequencing applications in mind, many references focused on the {\em consecutive ones problem} (C1P) which seeks to permute the rows of a binary matrix so that all the ones in each column are contiguous. In particular, \cite{Fulk65} studied further connections to interval graphs and~\cite{Kend71} crucially showed that a solution to  C1P can be obtained by solving the seriation problem on the squared data matrix. We refer the reader to \citep{Ding04,Vuok10,Liiv10} for a much more complete survey of applications.

On the algorithmic front, the seriation problem was shown to be NP-complete by \cite{Geor97}. Archeological examples are usually small scale and earlier references such as~\citep{Robi51} used greedy techniques to reorder matrices. Similar techniques were, and are still used to reorder genetic data sets. More general ordering problems were studied extensively in operations research, mostly in connection with the quadratic assignment problem (QAP), for which several convex relaxations were derived in e.g. \citep{Lawl63,Zhao98}. Since a matrix is a permutation matrix if and only if it is both orthogonal and doubly stochastic, much work also focused on producing semidefinite relaxations to orthogonality constraints \citep{Nemi07,So09}. These programs are convex and can be solved using conic programming solvers, but the relaxations are usually very large and scale poorly. More recently however, \cite{Atki98} produced a spectral algorithm that exactly solves the seriation problem in a noiseless setting. They show that for similarity matrices computed from serial variables (for which a total order exists), the ordering of the second eigenvector of the Laplacian (a.k.a.~the Fiedler vector) matches that of the variables, in results that are closely connected to those obtained on the interlacing of eigenvectors for Sturm Liouville operators. A lot of work has focused on the minimum linear arrangement problem or 1-SUM, with \citep{Even00,Feig00,Blum00} and \citep{Rao05,Feig07,Char10} producing semidefinite relaxations with nearly dimension independent approximation ratios. While these relaxations form semidefinite programs that have an exponential number of constraints, they admit a polynomial-time separation oracle and can be solved using the ellipsoid method. The later algorithm being extremely slow, these programs have very little practical impact. Finally, seriation is also directly related to the manifold learning problem \citep{Wein06} which seeks to reconstruct a low dimensional manifold based on local metric information. Seriation can be seen as a particular instance of that problem, where the manifold is unidimensional but the similarity information is not metric.

Our contribution here is twofold. First, we explicitly write seriation as an optimization problem by proving the equivalence between the seriation and combinatorial 2-SUM problems on similarity matrices. 2-SUM, defined in e.g.~\citep{Geor97}, is a quadratic minimization problem over permutations. Our result shows in particular that  2-SUM is polynomially solvable for matrices coming from serial data. This quadratic problem was mentioned in~\citep{Atki98}, but no explicit connection was established between combinatorial problems like 2-SUM and seriation. While this paper was under review, a recent working paper by \citep{Laur14} has extended the results in Propositions ~\ref{prop:q-r} and~\ref{prop:q-r-converse} here to show that the QAP problem $Q(A,B)$ is solved by the spectral seriation algorithm when $A$ is a similarity matrix (satisfying the Robinson assumption detailed below) and $B$ is a Toeplitz dissimilarity matrix (e.g. $B_{ij}=(i-j)^2$ in the 2-SUM problem discussed here).

Second, we derive several new convex relaxations for the seriation problem. Our simplest relaxation is written over the set of doubly stochastic matrices and appears to be more robust to noise than the spectral solution in a number of examples. Perhaps more importantly, it allows us to impose additional structural constraints to solve semi-supervised seriation problems. We also briefly outline a fast algorithm for projecting on the set of doubly stochastic matrices, which is of independent interest. In an appendix, we also produce a semidefinite relaxation for the seriation problem using the classical lifting argument in \citep{Shor87,Lova91} written on a nonconvex quadratic program (QP) formulation of the combinatorial 2-SUM problem. Based on randomization arguments in \citep{Nest98a,dAsp10} for the MaxCut and $k$-dense-subgraph problems, we show that this relaxation of the set of permutation {\em matrices} achieves an approximation ratio of $O(\sqrt{n})$. We also recall how several other relaxations of the minimum linear arrangement  (MLA) problem, written on permutation {\em vectors}, can be adapted to get nearly dimension independent $O(\sqrt{\log n})$ approximation ratios by forming (exponentially large but tractable) semidefinite programs. While these results are of limited practical impact because of the computational cost of the semidefinite programs they form, they do show that certain QAP instances written on Laplacian matrices, such as the seriation problem considered here, are much simpler to approximate than generic QAP problems. They also partially explain the excellent empirical performance of our relaxations in the numerical experiments of Section~\ref{s:apps}.

The paper is organized as follows. In Section~\ref{s:ser}, we show how to decompose similarity matrices formed in the C1P problem as conic combinations of CUT matrices, i.e. elementary block matrices. This allows us to connect the solutions of the seriation and 2-SUM minimization problems on these matrices. In Section~\ref{s:relax} we use these results to write convex relaxations of the seriation problem by relaxing the set of permutation matrices as doubly stochastic matrices in a QP formulation of the 2-SUM minimization problem. Section~\ref{s:algos} briefly discusses first order algorithms solving the doubly stochastic relaxation and details in particular a block coordinate descent algorithm for projecting on the set of doubly stochastic matrices. Finally, Section~\ref{s:apps} describes applications and numerical experiments on archeological data, Markov chains and DNA assembly problems. In the Appendix, we describe larger semidefinite relaxations of the 2-SUM QP and obtain $O(\sqrt{n})$ approximation bounds using randomization arguments. We also detail several direct connections with the minimum linear arrangement problem.

\subsection*{Notation.} We use the notation $\mathcal{P}$ for both the set of permutations of $\{1,\ldots,n\}$ and the set of permutation matrices. The notation $\pi$ will refer to a permuted vector $(1,\ldots,n)^T$ while the notation $\Pi$ (in capital letter) will refer to the corresponding matrix permutation, which is a $\{0,1\}$ matrix such that $\Pi_{ij}=1$ iff $\pi(i)=j$. Moreover $y_\pi$ is the vector with coefficients $(y_{\pi(1)},\ldots,y_{\pi(n)})$ hence $\Pi y=y_\pi$ and $\Pi^Ty_\pi=y$. This also means that $A\Pi^T$ is the matrix with coefficients $A_{i\pi(j)}$, and $\Pi A \Pi^T$ is the matrix with coefficients $A_{\pi(i)\pi(j)}$. For a vector $y\in\reals^n$, we write $\Var(y)$ its variance, with $\Var(y)=\sum_{i=1}^n y_i^2/n -(\sum_{i=1}^n y_i/n)^2$, we also write $y_{[u,v]}\in\reals^{v-u+1}$ the vector $(y_u,\ldots,y_v)^T$. Here, $e_i \in \reals^{n}$ is $i$-the Euclidean basis vector and $\ones$ is the vector of ones. Recall also that the matrix product can be written in terms of outer products, with $AB=\sum_{i} A_{(i)}B^{(i)}$, with $A_{(i)}$ (resp. $B^{(i)}$) the $i$-th column (resp. row) of $A$ (resp. $B$). For a matrix $A\in\reals^{m \times n}$, we write $\vect{A}\in\reals^{mn}$ the vector formed by stacking up the columns of $A$. We write $\idm$ the identity matrix and $\symm_n$ the set of symmetric matrices of dimension $n$, $\|\cdot\|_F$ denotes the Frobenius norm, $\lambda_i(X)$ the $i^\mathrm{th}$ eigenvalue (in increasing order) of $X$ and $\|X\|_\infty=\|\vect{X}\|_\infty$.

\section{Seriation, 2-SUM \& consecutive ones}\label{s:ser}

Given a symmetric, binary matrix $A$, we will focus on variations of the following 2-SUM combinatorial minimization problem, studied in e.g. \citep{Geor97}, and written
\BEQ\label{eq:2sum-pb}
\BA{ll}
\mbox{minimize} & \sum_{i,j=1}^n A_{ij}(\pi(i)-\pi(j))^2\\
\mbox{subject to} & \pi \in \mathcal{P},
\EA\EEQ
where $\mathcal{P}$ is the set of permutations of the vector $(1,\ldots,n)^T$. This problem is used for example to reduce the envelope of sparse matrices and is shown in \cite[Th.\,2.2]{Geor97} to be NP-complete. When $A$ has a specific structure, \cite{Atki98} show that a related matrix ordering problem used for seriation can be solved explicitly by a spectral algorithm. However, the results in~\cite{Atki98} do not explicitly link spectral ordering and the optimum of~\eqref{eq:2sum-pb}. The main objective of this section is to show the equivalence between the 2-SUM and seriation problems for certain classes of matrices $A$. In particular, for some instances of $A$ related to seriation and consecutive one problems, we will show below that the spectral ordering directly minimizes the objective of problem~\eqref{eq:2sum-pb}. We first focus on binary matrices, then extend our results to more general unimodal matrices.

Let $A\in\symm_n$ and consider the following generalization of the 2-SUM minimization problem 
\BEQ\label{eq:ser-pb}
\BA{ll}
\mbox{minimize} & f(y_\pi) \triangleq \sum_{i,j=1}^n A_{ij}(y_{\pi(i)}-y_{\pi(j)})^2\\
\mbox{subject to} & \pi \in \mathcal{P},
\EA\EEQ
in the permutation variable $\pi$, where $y\in\reals^n$ is a given weight vector. The classical 2-SUM minimization problem~\eqref{eq:2sum-pb} is a particular case of problem~\eqref{eq:ser-pb} with $y_i=i$. The main point of this section is to show that if $A$ is the permutation of a similarity matrix formed from serial data, then minimizing~\eqref{eq:ser-pb} recovers the correct variable ordering. To do this, we simply need to show that when $A$ is correctly ordered, a monotonic vector $y$ solves~\eqref{eq:ser-pb}, since reordering $y$ is equivalent to reordering $A$. Our strategy is to first show that we can focus on matrices $A$ that are sums of simple CUT matrices, i.e.~symmetric block matrices with a single constant block \citep[see][]{Frie99}. We then show that all minimization problems~\eqref{eq:ser-pb} written on CUT matrices have a common optimal solution, where $y_\pi$ is monotonic. 

\subsection{Similarity, C1P \& unimodal matrices}\label{ss:bin}

We begin by introducing a few definitions on R-matrices (i.e. similarity matrices), C1P and unimodal matrices following \citep{Atki98}.

\begin{definition}{\bf (R-matrices)}\label{def:r-mat}
We say that the matrix $A\in\symm_n$ is an R-matrix (or Robinson matrix) iff it is symmetric and satisfies
$A_{i,j} \leq  A_{i,j+1}$ and $A_{i+1,j} \leq A_{i,j}$ in the lower triangle, where $1\leq j < i \leq n$.
\end{definition}

Another way to write the R-matrix conditions is to impose $A_{ij}\leq A_{kl}$ if $|i-j|\leq |k-l|$ off-diagonal, i.e.~the coefficients of $A$ decrease as we move away from the diagonal (cf. Figure~\ref{fig:r-mat}). In that sense, \mbox{R-matrices} are similarity matrices between variables organized on a {\em chain}, i.e. where the similarity $A_{ij}$ is monotonically decreasing with the distance between $i$ and $j$ on this chain. We also introduce a few definitions related to the consecutive ones problem (C1P) and its unimodal extension.

\begin{figure}[thbp]
\begin{center}
\includegraphics[scale=0.8]{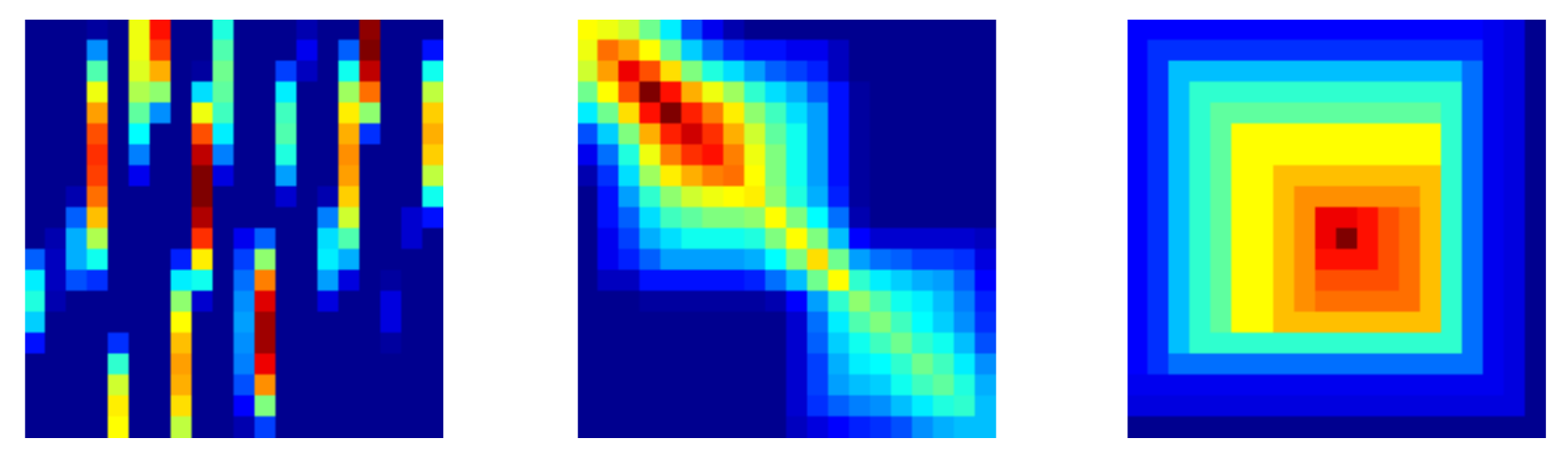}
\caption{A sample Q-matrix (see Def.~\ref{def:q-mat}), which has unimodal columns {\em (left)}, its ``circular square'' $A\circ A^T$ (see Def.~\ref{def:circ-prod}) which is an R-matrix {\em (center)}, and a matrix $a \circ a^T$ where $a$ is a unimodal vector {\em (right)}.\label{fig:r-mat}}
\end{center}
 \end{figure}

\begin{definition}{\bf (P-matrices)}\label{def:p-mat}
We say that the $\{0,1\}$-matrix $A\in\reals^{n \times m}$ is a  P-matrix (or Petrie matrix) iff for each column of $A$, the ones form a consecutive sequence.
\end{definition}

As in \citep{Atki98}, we will say that $A$ is {\em pre-R} (resp. {\em pre-P}) iff there is a permutation $\Pi$ such that $\Pi A \Pi^T$ is an R-matrix (resp.  $\Pi A$ is a P-matrix). Based on \cite{Kend71}, we also define a generalization of P-matrices called (appropriately enough) Q-matrices, i.e.~matrices with unimodal columns. 

\begin{definition}{\bf (Q-matrices)}\label{def:q-mat}
We say that a matrix $A\in\reals^{n \times m}$ is a Q-matrix if and only if each column of $A$ is unimodal, i.e.~the coefficients of each column increase to a maximum, then decrease.
\end{definition}

Note that R-matrices are symmetric Q-matrices. We call a matrix $A$ {\em pre-Q} iff there is a permutation $\Pi$ such that $\Pi A$ is a Q-matrix. Next, again based on \cite{Kend71}, we define the {\em circular product} of two matrices.

\begin{definition}\label{def:circ-prod}
Given $A,B^T\in\reals^{n \times m}$, their circular product $A \circ B$ is defined as
\[
(A \circ B)_{ij}=\sum_{k=1}^m \min\{A_{ik},B_{kj}\} \quad i,j=1,\ldots,n,
\]
note that when $A$ is a symmetric matrix, $A\circ A$ is also symmetric.
\end{definition}

Remark that when $A,B$ are $\{0,1\}$ matrices $\min\{A_{ik},B_{kj}\}=A_{ik}B_{kj}$, so the circular product matches the regular matrix product $AB$. Similarly, a $\{0,1\}$ matrix with the consecutive one property (C1P) is also unimodal. 

\subsection{Seriation on CUT matrices}\label{ss:cut}
We now introduce CUT matrices (named after the CUT decomposition in \citep{Frie99} whose definition is slightly more flexible), and first study the seriation problem on these simple block matrices. The motivation for this definition is that if $A$ is a $P$, $Q$ or $R$ matrix, then $A\circ A^T$ can we decomposed as a sum of CUT matrices. This is illustrated in Figure~\ref{fig:r-mat} and means that we can start by studying problem~\eqref{eq:ser-pb} on CUT matrices. 

\begin{definition}\label{def:cut-mat} For $u,v\in[1,n]$, we call $CUT(u,v)$ the matrix such that
\[
CUT(u,v)_{ij}=\left\{\BA{l}
1 \quad \mbox{if $u\leq i \leq v$ and $u\leq j \leq v$}\\
0 \quad \mbox{otherwise,}
\EA\right.
\]
i.e.~$CUT(u,v)$ is symmetric, block diagonal and has one square block equal to one.
\end{definition}

We first show that the objective of~\eqref{eq:ser-pb} has a natural interpretation when $A$ is a CUT matrix, as the variance of a subset of $y$ under a uniform probability measure.

\begin{lemma}\label{lem:var}
Suppose $A=CUT(u,v)$, then 
\[
f(y)=\sum_{i,j=1}^n A_{ij}(y_i-y_j)^2=(v-u+1)^2 \Var(y_{[u,v]}).
\]
\end{lemma}
\begin{proof}
We can write $\sum_{ij} A_{ij}(y_i-y_j)^2=y^T L_A y$ where $L_A=\diag(A\ones)-A$ is the Laplacian of~$A$, which is a block matrix with a single nonzero block equal to $(v-u+1)\delta_{\{i=j\}}-1$ for $u\leq i,j\leq v$.
\end{proof}

This last lemma shows that solving the seriation problem~\eqref{eq:ser-pb} for CUT matrices amounts to finding a subset of $y$ of size $(u-v+1)$ with minimum variance. This is the key to all the results that follow. As illustrated in Figure~\ref{fig:var}, for CUT matrices and of course conic combinations of CUT matrices, monotonic sequences have lower variance than sequences where the ordering is broken and the results that follow make this explicit. We now show a simple technical lemma about the impact of switching two coefficients in $y$ on the objective of problem~\eqref{eq:ser-pb}, when $A$ is a CUT matrix.

\begin{figure}[thbp]
\begin{center}
\includegraphics[scale=0.33]{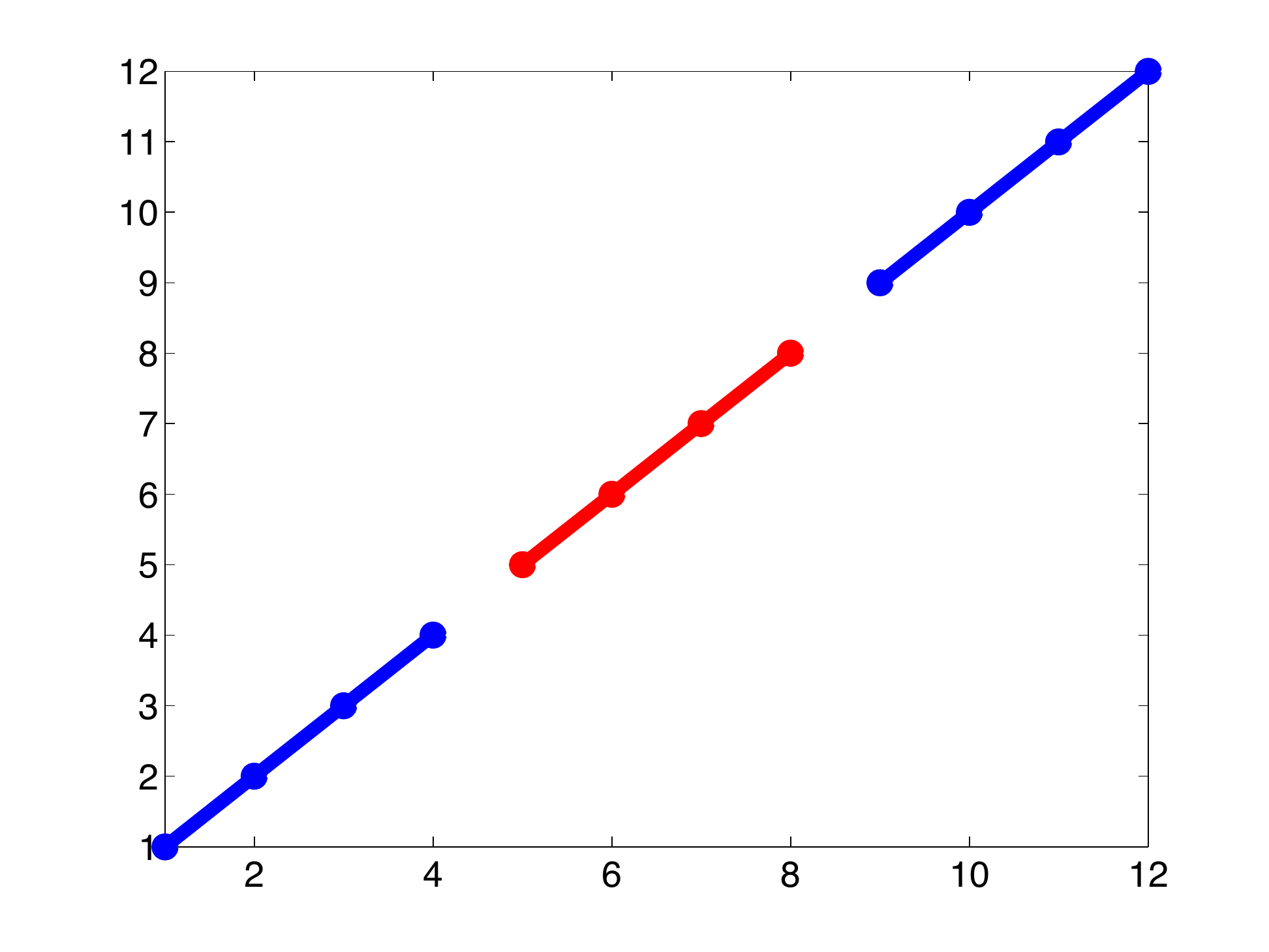}
\qquad
\includegraphics[scale=0.33]{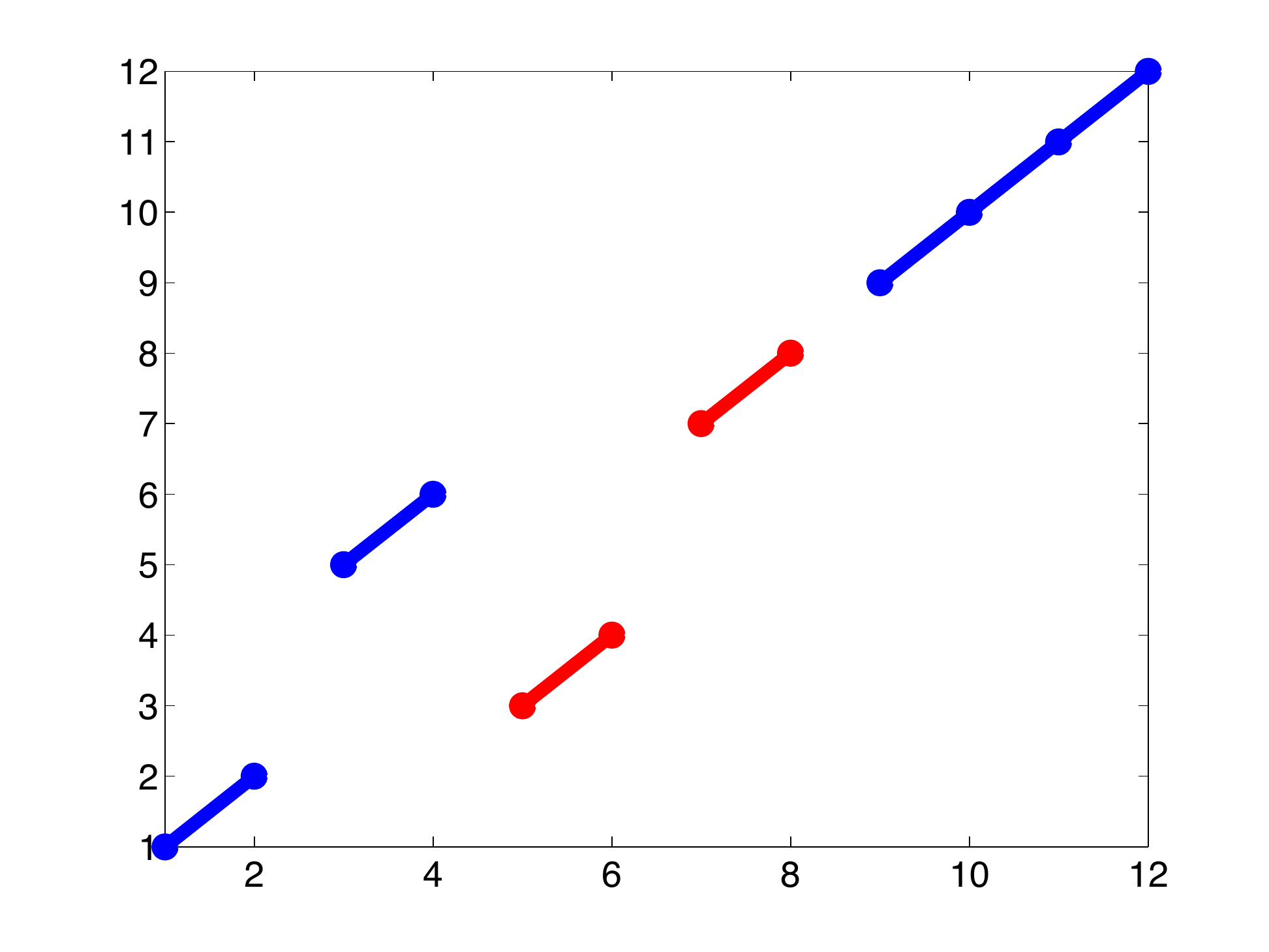}\\
\quad $\var(y_{[5,8]})=\mathbf{1.6}$ \qquad \qquad\qquad \qquad\qquad \qquad $\var(y_{[5,8]})=\mathbf{5.6}$
\caption{Objective values of 2-SUM problem~\eqref{eq:ser-pb} when $A=CUT(5,8)$ and $y_i=i$, $i=1,\ldots,12$. We plot the permuted values $y_{\pi(i)}$ against $i$, linking consecutive values of $y$ both inside and outside the interval $[5,8]$. The solution on the left, where the values of $y_{[5,8]}$ are consecutive, has $\var(y_{[5,8]})=1.6$ while $\var(y_{[5,8]})=5.6$ on the right, where there is a gap between $y_6$ and $y_7$. Minimizing the 2-SUM objective for CUT matrices, i.e. the variance of a subset of the coefficients of $y$, tends to pull the coefficients in this subset together. \label{fig:var}}
\end{center}
\end{figure}

\begin{lemma}\label{lem:switch}
Let $A\in\symm_n$, $y\in\reals^n$ and $f(\cdot)$ be the objective of problem~\eqref{eq:ser-pb}. Suppose we switch the values of $y_j$ and $y_{j+1}$ calling the new vector $z$, we have
\[
f(y)-f(z)=4\sum_{\substack{i=1\\i \neq j,\, i\neq j+1}}^n \left(\frac{y_j+y_{j+1}}{2}-y_i\right)(y_{j+1}-y_j)(A_{ij+1}-A_{ij}).
\]
\end{lemma}
\begin{proof}
Because $A$ is symmetric, we have 
\BEAS
(f(y)-f(z))/2&=&\sum_{i \neq j,\, i\neq j+1} A_{ij}(y_i-y_j)^2+ \sum_{i \neq j,\, i\neq j+1} A_{ij+1}(y_i-y_{j+1})^2\\
&& - \sum_{i \neq j,\, i\neq j+1} A_{ij}(y_i-y_{j+1})^2- \sum_{i \neq j,\, i\neq j+1} A_{ij+1}(y_i-y_j)^2\\
&=& \sum_{i \neq j,\, i\neq j+1} 2A_{ij}(y_j-y_{j+1}) \left(\frac{y_j+y_{j+1}}{2}-y_i\right)\\
&& + \sum_{i \neq j,\, i\neq j+1} 2A_{ij+1}(y_{j+1}-y_j) \left(\frac{y_j+y_{j+1}}{2}-y_i\right),
\EEAS
which yields the desired result.
\end{proof}

The next lemma characterizes optimal solutions of problem~\eqref{eq:ser-pb} for CUT matrices and shows that they split the coefficients of $y$ in disjoint intervals. 

\begin{lemma}\label{lem:contig}
Suppose $A=CUT(u,v)$, and write $w=y_\pi$ the optimal solution to~\eqref{eq:ser-pb}. If we call $I=[u,v]$ and $I^c$ its complement, then
\[
w_j \notin [\min(w_I),\max(w_I)], \quad \mbox{for all $j\in I^c$},
\]
in other words, the coefficients in $w_I$ and $w_{I^c}$ belong to disjoint intervals.
\end{lemma}
\begin{proof}
Without loss of generality, we can assume that the coefficients of $w_I$ are sorted in increasing order. By contradiction, suppose that there is a $w_j$ such that $j\in I^c$ and $w_j \notin [w_u,w_v]$.  Suppose also that $w$ is larger than the mean of coefficients inside I, i.e.~$w_j \geq \sum_{i=u+1}^v w_i/(v-u)$. This, combined with our assumption that $w_j \leq w_v$ and Lemma~\ref{lem:switch} means that switching the values of~$w_j$ and $w_v$ will decrease the objective by
\[
4\sum_{i=u}^{v-1} \left(\frac{w_j+w_{v}}{2}-y_i\right)(w_v-w_j)
\]
which is positive by our assumptions on $w_j$ and the mean which contradicts optimality. A symmetric result holds if $w_j$ is smaller than the mean.
\end{proof}

This last lemma shows that when $A$ is a CUT matrix, then the monotonic vector $y_i=a i +b$, for $a,b \in \reals$ and $i=1,\ldots,n,$ is always an optimal solution to the 2-SUM problem~\eqref{eq:ser-pb}, since all subvectors of~$y$ of a given size have the same variance. This means that, when $y$ is a permutation of $y_i=a i +b$, all minimization problems~\eqref{eq:ser-pb} written on CUT matrices have {\em a common optimal solution}, where $y_\pi$ is monotonic. 

\subsection{Ordering P, Q \& R matrices}\label{ss:pqr-order}
Having showed that all 2-SUM problems~\eqref{eq:ser-pb} written on CUT matrices share a common monotonic solution, this section now shows how to decompose the square of P, Q and R-matrices as a sum of CUT matrices, then links the reordering of a matrix with that of its square $A\circ A^T$. We will first show a technical lemma proving that if $A$ is a Q-matrix, then $A\circ A^T$ is a conic combination of CUT matrices. The CUT decomposition of P and R-matrices will then naturally follow, since P-matrices are just $\{0,1\}$ Q-matrices, and R-matrices are symmetric Q-matrices. 

\begin{lemma}\label{lem:circ-cut}
Suppose $A\in\reals^{n \times m}$ is a Q-matrix, then $A\circ A^T$ is a conic combination of CUT matrices.
\end{lemma}
\begin{proof}
Suppose, $a\in\reals^n$ is a unimodal vector, let us show that the matrix $M=a \circ a^T$ with coefficients $M_{ij}=\min\{a_i,a_j\}$ is a conic combination of CUT matrices. Let $I=\{j: j=\argmax_i a_i \}$, then $I$ is an index interval $[I_\mathrm{min},I_\mathrm{max}]$ because $a$ is unimodal. Call $\bar a=\max_i a_i$ and $b=\max_{i \in I^c} a_i$ (with $b=0$ when $I^c=\emptyset$), the deflated matrix
\[
M^-=M - (\bar a - b)~ CUT(I_\mathrm{min},I_\mathrm{max})
\]
can be written $M^-=a^- \circ (a^-)^T$ with
\[
a^-=a-(\bar a - b) v
\]
where $v_i=1$ iff $i\in I$. By construction $|\argmax M^-|>|I|$, i.e.~the size of $\argmax M$ increases by at least one, so this deflation procedure ends after at most $n$ iterations. This shows that $a \circ a^T$ is a conic combination of CUT matrices when $a$ is unimodal. Now, we have $(A\circ A^T)_{ij}=\sum_{k=1}^n w_k \min\{A_{ik},A_{jk}\}$, so $A\circ A^T$ is a sum of $n$ matrices of the form $\min\{A_{ik},A_{jk}\}$ where each column is unimodal, hence the desired result.
\end{proof}

This last result also shows that, when $A$ is a Q matrix, $A \circ A^T$ is a R-matrix as a sum of CUT matrices, which is illustrated in Figure~\ref{fig:r-mat}. We now recall the central result in \cite[Th.\,1]{Kend71} showing that for Q-matrices, reordering $A\circ A^T$ also reorders $A$.

\begin{theorem}\label{th:pq}
{\bf \cite[Th.\,1]{Kend71}} Suppose $A\in\reals^{n \times m}$ is pre-Q, then $\Pi A$ is a Q-matrix if and only if $\Pi (A\circ A^T) \Pi^T$ is a R-matrix.
\end{theorem}

%

We use these last results to show that at least for some vectors~$y$, if $C$ is a Q-matrix then the 2-SUM problem~\eqref{eq:ser-pb} written on $A=C \circ C^T$ has a monotonic solution $y_\pi$.

\begin{proposition}\label{prop:q-r}
Suppose $C\in\reals^{n \times m}$ is a pre-Q matrix and $y_i=a i +b$ for $i=1,\ldots,n$ and $a,b\in\reals$ with $a\neq 0$. Let $A=C \circ C^T$, if $\Pi$ is such that $\Pi A \Pi^T$ is an R-matrix, then the corresponding permutation~$\pi$ solves the combinatorial minimization problem~\eqref{eq:ser-pb}.
\end{proposition}
\begin{proof}
If $C\in\reals^{n \times m}$ is pre-Q, then Lemma~\ref{lem:circ-cut} and Theorem~\ref{th:pq} show that there is a permutation $\Pi$ such that $\Pi (C \circ C^T) \Pi^T$ is a sum of CUT matrices (hence a R-matrix). Now all monotonic subsets of $y$ of a given length have the same variance, hence Lemmas~\ref{lem:var} and \ref{lem:contig} show that $\pi$ solves problem~\eqref{eq:ser-pb}.
\end{proof}

We now show that when the R-constraints are strict, the converse is also true, i.e. for matrices that are the square of Q-matrices, if $y_\pi$ solves the \mbox{2-SUM} problem~\eqref{eq:ser-pb}, then $\pi$ makes $A$ an R-matrix. In the next section, we will use this result to reorder \mbox{pre-R} matrices (with noise and additional structural constraints) by solving convex relaxations to the \mbox{2-SUM} problem.

\begin{proposition}\label{prop:q-r-converse}
Suppose $A$ is a pre-R matrix that can be written as $A=C \circ C^T$, where $C\in\reals^{n \times m}$ is a pre-Q matrix, $y_i=a i +b$ for $i=1,\ldots,n$ and $a,b\in\reals$ with $a\neq 0$. Suppose moreover that $A$ has strict R-constraints, i.e.~the rows/columns of $A$ are strictly unimodal after reordering.  If the permutation~$\pi$ solves the 2-SUM problem~\eqref{eq:ser-pb}, then the corresponding permutation matrix $\Pi$ is such that $\Pi A \Pi^T$ is an R-matrix.
\end{proposition}
\begin{proof}
We can assume that $A$ is a R-matrix without loss of generality. We will show that the identity is optimal for 2-SUM and that it is the unique such solution, hence solving 2-SUM solves seriation. Lemma~\ref{lem:circ-cut} shows that $A$ is a conic combination of CUT matrices. Moreover, by Proposition~\ref{prop:q-r} the identity matrix solves problem~\eqref{eq:ser-pb}. Following the proof of Proposition~\ref{prop:q-r}, the identity matrix is also optimal for each seriation subproblem on the CUT matrices of~$A$. 

Now remark that since the R-constraints are strict on the first column of $A$, there must be $n-2$ CUT matrices of the form $A_i=CUT(1,i)$ for $i=2,\ldots,n-1$ in the decomposition of $A$ (otherwise, there would be some index $k>1$ for which $A_{1k}=A_{1k+1}$ which would contradict our strict unimodal assumption).
Following the previous remarks, the identity matrix is optimal for all the seriation subproblems in $A_i$, which means that the variance of all the corresponding subvectors of $y_\pi$, i.e. $(y_{\pi(1)},y_{\pi(2)})$, $(y_{\pi(1)},y_{\pi(2)},y_{\pi(3)})$,\ldots, $(y_{\pi(1)},\ldots,y_{\pi(n-1)})$ must be minimized. Since these subvectors of $y_\pi$ are monotonically embedded, up to a permutation of $y_{\pi(1)}$ and $y_{\pi(2)}$, Lemma~\ref{lem:contig} shows that this can only be achieved for contiguous $y_{\pi(i)}$, that is for $\pi$ equal to the identity or the reverse permutation. Indeed, to minimize the variance of $(y_{\pi(1)},\ldots,y_{\pi(n-1)})$, we have to choose $\pi(n)=n$ or $\pi(n)=1$. Then to minimize the variance of $(y_{\pi(1)},\ldots,y_{\pi(n-2)})$, we have to choose respectively $\pi(n-1)=n-1$ or $\pi(n-1)=2$. Thus we get by induction respectively $\pi(i)=i$ or $\pi(i)=n-i+1$ for $i=3,\ldots,n$. Finally, there are only two permutations left for $y_{\pi(1)}$ and $y_{\pi(2)}$. Since $A_{31}<A_{32}$, we have to choose $(y_{\pi(3)}-y_{\pi(1)})^2>(y_{\pi(3)}-y_{\pi(2)})^2$, and the remaining ambiguity on the order of $y_{\pi(1)}$ and $y_{\pi(2)}$ is removed.
\end{proof}

These results shows that if $A$ is pre-R and can be written $A=C \circ C^T$ with $C$ pre-Q, then the permutation that makes $A$ an R-matrix also solves the 2-SUM problem~\eqref{eq:ser-pb}. Conversely, when $A$ is pre-R (strictly), the permutation that solves~\eqref{eq:ser-pb} reorders $A$  as a R-matrix. Since \cite{Atki98} show that sorting the Fiedler vector also orders A as an R-matrix, Proposition~\ref{prop:q-r} gives a polynomial time solution to the 2-SUM problem~\eqref{eq:ser-pb} when $A$ is pre-R with $A=C \circ C^T$ for some pre-Q matrix $C$. Note that the strict monotonicity constraints on the R-matrix can be somewhat relaxed (we only need one strictly monotonic column plus two more constraints), but requiring strict monotonicity everywhere simplifies the argument.

\section{Convex relaxations}\label{s:relax}
In the sections that follow, we will use the combinatorial results derived above to produce convex relaxations of optimization problems written over the set of permutation matrices. We mostly focus on the 2-SUM problem in~\eqref{eq:ser-pb}, however many of the results below can be directly adapted to other objective functions. We detail several convex approximations, some new, some taken from the computer science literature, ranked by increasing numerical complexity. Without loss of generality, we always assume that the weight matrix $A$ is nonnegative (if $A$ has negative entries, it can be shifted to become nonnegative, with no impact on the permutation problem). The nonnegativity assumption is in any case natural since $A$ represents a similarity matrix in the seriation problem.

\subsection{Spectral ordering}\label{ss:spectral}
We first recall classical definitions from spectral clustering and briefly survey the spectral ordering results in \citep{Atki98} in the noiseless setting.

\begin{definition}\label{def:fiedler}
The Fiedler value of a symmetric, nonnegative matrix $A$ is the smallest non-zero eigenvalue of its Laplacian $L_A=\diag(A\ones)-A$. The corresponding eigenvector is called Fiedler vector and is the optimal solution to
\BEQ\label{eq:fiedler}
\BA{ll}
\mbox{$\mathrm{minimize}$} & y^T  L_A  y\\
\mbox{$\mathrm{subject \ to}$} & y^T\ones=0, \|y\|_2=1.
\EA\EEQ
in the variable $x\in\reals^n$.
\end{definition}

We now recall the main result from \citep{Atki98} which shows how to reorder pre-R matrices in a noise free setting.  

\begin{proposition}
{\bf \cite[Th.\,3.3]{Atki98}} Suppose $A\in\symm_n$ is a pre-R-matrix, with a simple Fiedler value whose Fiedler vector $v$ has no repeated values. Suppose that $\Pi$ is a permutation matrix such that the permuted Fielder vector $\Pi v$ is strictly monotonic, then $\Pi  A \Pi^T$ is an R-matrix.
\end{proposition}

We now extend the result of Proposition~\ref{prop:q-r} to the case where the weights $y$ are given by the Fiedler vector. 

\begin{proposition}\label{prop:fiedlerExtension}
Suppose $A\in\symm^{n \times n}$ is a R-matrix and $y$ is its Fiedler vector. Then the identity permutation solves the 2-SUM problem~\eqref{eq:ser-pb}.
\end{proposition}
\begin{proof} The combinatorial problem~\eqref{eq:ser-pb} can be rewritten 
\[
\BA{ll}
\mbox{minimize} & y^T \Pi^T L_A  \Pi y\\
\mbox{subject to} & \Pi\in\mathcal{P},
\EA
\]
which is also equivalent to
\[
\BA{ll}
\mbox{minimize} & z^T L_A  z\\
\mbox{subject to} & z^T\ones=0, \|z\|_2=1, z=\Pi y, \Pi\in\mathcal{P},
\EA
\]
since $y$ is the Fiedler vector of $A$. By dropping the constraints $ z=\Pi y, \Pi\in\mathcal{P}$, we can relax the last problem into~\eqref{eq:fiedler}, whose solution is the Fiedler vector of $A$. Note that the optimal value of problem~\eqref{eq:ser-pb} is thus an upper bound on that of its relaxation~\eqref{eq:fiedler}, i.e.~the Fiedler value of $A$. This upper bound is attained by the Fiedler vector, i.e. the optimum of~\eqref{eq:fiedler}, therefore the identity matrix is an optimal solution to~\eqref{eq:ser-pb}. 
\end{proof}

Using the fact that the Fiedler vector of a R-matrix is monotonic \citep[Th.\,3.2]{Atki98}, the next corollary immediately follows.

\begin{corollary}
If $A$ is a pre-R matrix such that $\Pi^T A \Pi$ is a R-matrix, then $\pi$ is an optimal solution to problem~\eqref{eq:ser-pb} when $y$ is the Fiedler vector of $A$ sorted in increasing order.
\end{corollary}

The results in \citep{Atki98} thus provide a polynomial time solution to the R-matrix ordering problem in a noise free setting (extremal eigenvalues of dense matrices can be computed by randomized polynomial time algorithms with complexity $O(n^2\log n)$ \citep{Kucz92}). While \cite{Atki98} also show how to handle cases where the Fiedler vector is degenerate, these scenarios are highly unlikely to arise in settings where observations on $A$ are noisy and we refer the reader to \cite[\S4]{Atki98} for details.

\subsection{QP relaxation}\label{ss:qp}
In most applications, $A$ is typically noisy and the pre-R assumption no longer holds.  The spectral solution is stable when the magnitude of the noise remains within the spectral gap (i.e., in a perturbative regime \citep{Stew90}). Beyond that, while the Fiedler vector of $A$ can still be used as a heuristic to find an approximate solution to~\eqref{eq:ser-pb}, there is no guarantee that it will be optimal.

The results in Section~\ref{s:ser} made the connection between the spectral ordering in \citep{Atki98} and the 2-SUM problem~\eqref{eq:ser-pb}. In what follows, we will use convex relaxations to~\eqref{eq:ser-pb} to solve matrix ordering problems in a noisy setting. We also show in \S\ref{ss:struct} how to incorporate a priori knowledge on the true ordering in the formulation of the optimization problem to solve semi-supervised seriation problems. Numerical experiments in Section~\ref{s:apps} show that semi-supervised seriation solutions are sometimes significantly more robust to noise than the spectral solutions ordered from the Fiedler vector.

\subsubsection{Permutations and doubly stochastic matrices}
We write $\mathcal{D}_n$ the set of doubly stochastic matrices, i.e. $\mathcal{D}_n = \{ X \in \mathbb{R}^{n\times n} :  X\geqslant 0, X\mathbf{1}=\mathbf{1}, X^T\mathbf{1}=\mathbf{1} \}$. Note that $\mathcal{D}_n$ is convex and polyhedral. Classical results show that the set of doubly stochastic matrices is the convex hull of the set of permutation matrices.
We also have $\mathcal{P}=\mathcal{D}\cap\mathcal{O}$, i.e.~a matrix is a permutation matrix if and only if it is both doubly stochastic and orthogonal. The fact that $L_A \succeq 0$ means that we can directly write a convex relaxation to the combinatorial problem~\eqref{eq:ser-pb} by replacing $\mathcal{P}$ with its convex hull $\mathcal{D}_n$, to get
\BEQ\label{eq:relaxedPb1}
\BA{ll}
\mbox{minimize} & g^T \Pi^T L_A \Pi g\\
\mbox{subject to} & \Pi \in \mathcal{D}_n,
\EA\EEQ
where $g=(1,\ldots,n)$, in the permutation matrix variable $\Pi\in\mathcal{P}$. By symmetry, if a vector $\Pi y$ minimizes~\eqref{eq:relaxedPb1}, then the reverse vector also minimizes~\eqref{eq:relaxedPb1}. This often has a significant negative impact on the quality of the relaxation, and we add the linear constraint $e_1^T \Pi g + 1 \leq e_n^T \Pi g$ to break symmetries, which means that we always pick solutions where the first element comes before the last one.
Because the Laplacian $L_A$ is positive semidefinite, problem~\eqref{eq:relaxedPb1} is a convex quadratic program in the variable $\Pi\in \mathbb{R}^{n\times n}$ and can be solved efficiently. To produce approximate solutions to problem~\eqref{eq:ser-pb}, we then generate permutations from the doubly stochastic optimal solution to the relaxation in~\eqref{eq:relaxedPb1} (we will describe an efficient procedure to do so in \S\ref{sec:sampling-proc}). 

The results of Section~\ref{s:ser} show that the optimal solution to~\eqref{eq:ser-pb} also solves the seriation problem in the noiseless setting when the matrix $A$ is of the form $C \circ C^T$ with $C$ a Q-matrix and $y$ is an affine transform of the vector $\left(1,\ldots,n\right)$. These results also hold empirically for small perturbations of the vector $y$ and to improve robustness to noisy observations of $A$, we average several values of the objective of~\eqref{eq:relaxedPb1} over these perturbations, solving
\BEQ\label{eq:relaxedPb2}
\BA{ll}
\mbox{minimize} & \Tr (Y^T\Pi^T L_A\Pi Y)/p \\
\mbox{subject to} &  e_1^T \Pi g + 1 \leq e_n^T \Pi g,\\
& \Pi \ones=\ones,\, \Pi^T\ones=\ones, \, \Pi \geq 0,
\EA\EEQ
in the variable $\Pi \in\reals^{n\times n}$, where $Y \in \reals^{n\times p}$ is a matrix whose columns are small perturbations of the vector $g=(1,\ldots,n)^T$. Solving~\eqref{eq:relaxedPb2} is roughly $p$ times faster than individually solving $p$ versions of~\eqref{eq:relaxedPb1}.

\subsubsection{Regularized QP relaxation} \label{sss:regularized}
In the previous section, we have relaxed the combinatorial problem~\eqref{eq:ser-pb} by relaxing the set of permutation matrices into the set of doubly stochastic matrices. As the set of permutation matrices $\mathcal{P}$ is the intersection of the set of doubly stochastic matrices $\mathcal{D}$ and the set of orthogonal matrices $\mathcal{O}$, i.e.~$\mathcal{P}=\mathcal{D}\cap\mathcal{O}$ we can add a penalty to the objective of the convex relaxed problem~\eqref{eq:relaxedPb2} to force the solution to get closer to the set of orthogonal matrices. Since a doubly stochastic matrix of Frobenius norm $\sqrt{n}$ is necessarily orthogonal, we would ideally like to solve

\BEQ\label{eq:relaxedPbRegIdeal}
\BA{ll}
\mbox{minimize} &  \frac{1}{p} \Tr(Y^T\Pi^T L_A\Pi Y) - \frac{\mu}{p}\| \Pi\|_F^2\\
\mbox{subject to} &  e_1^T \Pi g + 1 \leq e_n^T \Pi g,\\
& \Pi \ones=\ones,\, \Pi^T\ones=\ones, \, \Pi \geq 0,
\EA\EEQ
with $\mu$ large enough to guarantee that the global solution is indeed a permutation. However, this problem is not convex for any $\mu >0$ since its Hessian is not positive semi-definite. Note that the objective of~\eqref{eq:relaxedPb2} can be rewritten as $\mathrm{Vec}(\Pi)^T(YY^T\otimes L_A)\mathrm{Vec}(\Pi)/p$ so the Hessian here is $YY^T\otimes L_A - \mu I \otimes I$ and is never positive semidefinite when $\mu>0$ since the first eigenvalue of $L_A$ is always zero. Instead, we propose a slightly modified version of~\eqref{eq:relaxedPbRegIdeal}, which has the same objective function up to a constant, and is convex for some values of $\mu$. Recall that the Laplacian matrix $L_A$ is always positive semidefinite with at least one eigenvalue equal to zero corresponding to the eigenvector $\ones/\sqrt{n}$ (strictly one if the graph is connected) and let $P=\idm - \frac{1}{n}\mathbf{1}\mathbf{1}^T$. 

\begin{proposition}\label{prop:opt}
The optimization problem
\BEQ \label{eq:relaxedPbReg1}
\BA{ll}
\mbox{$\mathrm{minimize}$} &  \frac{1}{p}\Tr(Y^T\Pi^T L_A\Pi Y) - \frac{\mu}{p}\|P\Pi\|_F^2\\
\mbox{$\mathrm{subject \ to}$} &  e_1^T \Pi g + 1 \leq e_n^T \Pi g,\\
& \Pi \ones=\ones,\, \Pi^T\ones=\ones, \, \Pi \geq 0,
\EA\EEQ
is equivalent to problem~\eqref {eq:relaxedPbRegIdeal}, their objectives differ by a constant. Furthermore, when $\mu \leq \lambda_2(L_A) \lambda_1(YY^T)$, this problem is convex.
\end{proposition}
\begin{proof}
Let us first remark that 
\BEAS
\|P\Pi\|_F^2 & = & \Tr(\Pi^T P^T P \Pi)  = \Tr(\Pi^T P\Pi)\\
& = & \Tr(\Pi^T (I - \mathbf{1}\mathbf{1}^T/n)  \Pi) = \Tr(\Pi^T\Pi -  \mathbf{1}\mathbf{1}^T/n) ) \\
& = & \Tr(\Pi^T \Pi) - 1
\end{eqnarray*}
where we used the fact that $P$ is the (symmetric) projector matrix onto the orthogonal complement of $\mathbf{1}$ and $\Pi$ is doubly stochastic (so $\Pi \mathbf{1}=\Pi^T \mathbf{1}=\mathbf{1} $).
We deduce that problem~\eqref{eq:relaxedPbReg1} has the same objective function as~\eqref{eq:relaxedPbRegIdeal} up to a constant. Moreover, it is convex when $\mu \leq \lambda_2(L_A) \lambda_1(YY^T)$ since the Hessian of the objective is given by 
\BEQ\label{eq:def-A}
-\mathcal{A}=\frac{1}{p}\,YY^T\otimes L_A - \frac{\mu}{p}\, \idm \otimes P
\EEQ
and the eigenvalues of $YY^T\otimes L_A$, which are equal to $\lambda_i(L_A) \lambda_j(YY^T)$ for all $i,j$ in $\{1,\ldots,n\}$ are all superior or equal to the eigenvalues of $\mu\idm \otimes P$ which are all smaller than $\mu$.
\end{proof}

To have $\mu$ strictly positive, we need $YY^T$ to be definite, which can be achieved w.h.p. by setting $p$ higher than~$n$ and sampling independent vectors $y$. The key motivation for including several monotonic vectors $y$ in the objective of~\eqref{eq:relaxedPbReg1} is to increase the value of $\lambda_1(YY^T)$. The higher this eigenvalue, the stronger the effect of regularization term in (7), which in turn improves the quality of the solution (all of this being somewhat heuristic of course). The problem of generating good matrices $Y$ with both monotonic columns and high values of $\lambda_1(YY^T)$ is not easy to solve however, hence we use randomization to generate $Y$.

\subsubsection{Semi-supervised problems}\label{ss:struct}
The QP relaxation above allows us to add structural constraints to the problem. For instance, in archeological applications, one may specify that observation $i$ must appear before observation $j$, i.e.~$\pi(i) < \pi(j)$. In gene sequencing applications, one may constrain the distance between two elements (e.g.~mate reads), which would be written $a\leq \pi(i)-\pi(j) \leq b$ and introduce an affine inequality on the variable $\Pi$ in the QP relaxation of the form $a \leq e_i ^T \Pi g - e_j^T\Pi g \leq b$. Linear constraints could also be extracted from a reference gene sequence. More generally, we can rewrite problem~\eqref{eq:relaxedPbReg1} with~$n_c$ additional linear constraints as follows
\BEQ\label{eq:relaxedPbRegAddCons}
\BA{ll}
\mbox{minimize} &  \frac{1}{p}\Tr(Y^T\Pi^T L_A\Pi Y) - \frac{\mu}{p}\|P\Pi\|_F^2\\
\mbox{subject to} &  D^T \Pi g + \delta \leq 0,\\
& \Pi \ones=\ones,\, \Pi^T\ones=\ones, \, \Pi \geq 0,
\EA\EEQ
where $D$ is a matrix of size $n \times (n_c+1)$ and $\delta$ is a vector of size $n_c$. The first column of $D$ is equal to $ e_1 - e_n$ and $\delta_1=1$ (to break symmetry).

\subsubsection{Sampling permutations from doubly stochastic matrices}\label{sec:sampling-proc}
This procedure is based on the fact that a permutation can be defined from a doubly stochastic matrix~$S$ by the order induced on a monotonic vector. A similar argument was used in \citep{Barv06} to round orthogonal matrices into permutations. Suppose we generate a {\em monotonic} random vector~$v$ and compute $Sv$. To each $v$, we can associate a permutation $\Pi$ such that $\Pi Sv$ is monotonically increasing. If $S$ is a permutation matrix, then the permutation $\Pi$ generated by this procedure will be constant, if $S$ is a doubly stochastic matrix but not a permutation, it might fluctuate. Starting from a solution $S$ to problem~\eqref{eq:relaxedPbReg1}, we can use this procedure to sample many permutation matrices $\Pi$ and we pick the one with lowest cost $g^T\Pi^TL_A\Pi g$ in the combinatorial problem~\eqref{eq:ser-pb}. We could also project $S$ on permutations using the Hungarian algorithm, but this proved more costly and less effective in our experiments.

\section{Algorithms}\label{s:algos}
The convex relaxation in~\eqref{eq:relaxedPbRegAddCons} is a quadratic program in the variable $\Pi\in\reals^{n\times n}$, which has dimension~$n^2$. For reasonable values of $n$ (around a few hundreds), interior point solvers such as MOSEK~\citep{Ande00} solve this problem very efficiently (the experiments in this paper were performed using this library). Furthermore, most pre-R matrices formed by squaring pre-Q matrices are very sparse, which considerably speeds up linear algebra. However, first-order methods remain the only alternative for solving~\eqref{eq:relaxedPbRegAddCons} beyond a certain scale. We quickly discuss below the implementation of two classes of methods: the conditional gradient (a.k.a.~Frank-Wolfe) algorithm, and accelerated gradient methods. Alternatively, \citep{Goem09} produced an extended formulation of the permutahedron using only $O(n\log n)$ variables and constraints, which can be used to write QP relaxations of 2-SUM with only $O(n\log n)$ variables. While the constant in these representations is high, more practical formulations are available with $O(n\log^2 n)$ variables. This formulation was tested by \citep{Lim14} while this paper was under review, and combined with an efficient interior point solver (GUROBI) provides significant speed-up.

\subsection{Conditional gradient} Solving~\eqref{eq:relaxedPbRegAddCons} using the conditional gradient algorithm in e.g.~\citep{Fran56} requires minimizing an affine function over the set of doubly stochastic matrices at each iteration. This amounts to solving a classical transportation (or matching) problem for which very efficient solvers exist~\citep{Port96}.

\subsection{Accelerated smooth optimization} On the other hand, solving~\eqref{eq:relaxedPbRegAddCons} using accelerated gradient algorithms requires solving a projection step on doubly stochastic matrices at each iteration~\citep{Nest03a}. Here too, exploiting structure significantly improves the complexity of these steps. Given some matrix $\Pi_0$, the Euclidean projection problem is written
\BEQ\label{eq:primalProjection}
\BA{ll}
\mbox{minimize} &  \frac{1}{2} \| \Pi - \Pi_0  \|_F ^2\\
\mbox{subject to} &  D^T \Pi g + \delta \leq 0,\\
& \Pi \ones=\ones,\, \Pi^T\ones=\ones, \, \Pi \geq 0
\EA\EEQ
in the variable $\Pi\in\reals^{n\times n}$, with parameter $g\in\reals^n$. The dual is written
\BEQ\label{eq:dualProjection}
\BA{ll}
\mbox{maximize} & -\frac{1}{2} \| x \ones^T + \ones y^T + D z g^T - Z \|_F^2
-\Tr(Z^T\Pi_0)\\ 
&+ x^T ( \Pi_0 \ones - \ones ) + y^T ( \Pi_0^T \ones - \ones ) 
+ z ( D^T \Pi_0 g + \delta ) \\
\mbox{subject to} &  z \geq 0, \, Z \geq 0
\EA
\EEQ
in the variables $Z\in\reals^{n\times n}$, $x,y\in\reals^n$ and $z\in\reals^{n_c}$. The dual optimizes over decoupled linear constraints in $(z,\ Z)$, while $x$ and $y$ are unconstrained.

Each subproblem is equivalent to computing a conjugate norm and can be solved in closed form. This means that, with independent constraints ($D$ full rank), at each iteration, explicit formulas are available to update variables block by block in the dual Euclidean projection problem~\eqref{eq:dualProjection} over doubly stochastic matrices (\cf~Algorithm~\ref{alg:BCAPDS}).
Problem~\eqref{eq:dualProjection} can thus be solved very efficiently by block-coordinate ascent, whose convergence is guaranteed in this setting~\citep{Bert98}, and a solution to~\eqref{eq:primalProjection} can be reconstructed from the optimum in~\eqref{eq:dualProjection}. 

The detailed procedure for block coordinate ascent in the dual Euclidean projection problem~\eqref{eq:dualProjection} is described in Algorithm~\ref{alg:BCAPDS}. We perform block coordinate ascent until the duality gap between the primal and the dual objective is below the required precision. Warm-starting the projection step in both primal and dual provided a very significant speed-up in our experiments.

\begin{algorithm}[ht]
\caption{Projection on doubly stochastic matrices.}
\label{alg:BCAPDS}
\begin{algorithmic} [1]
\REQUIRE A matrix $Z \in\reals_+^{n \times n}$, a vector $z\in\reals_+^{n_c}$, two vectors $x, y \in\reals^{n}$, a target precision $\epsilon$, a maximum number of iterations $N$.
\STATE Set $k=0$.
\WHILE{$\mbox{duality gap}> \epsilon$ \& $k \leq N$}
\STATE Update dual variables
 \[\left\{\BA{l}
Z = \max\{ {\mathbf 0},\  x \ones^T + \ones y^T + D z g^T - \Pi_0 \}\\
x =  \frac{1}{n} ( \Pi_0 \ones - (y^T\ones+1) \ones - D z g^T\ones + Z\ones)\\
y =  \frac{1}{n} ( \Pi_0^T \ones - (x^T\ones+1) \ones - g z^T D\ones + Z^T\ones)\\
z =  \frac{1}{\|g\|_2^2} \max\{0,\ (D^TD)^{-1}( D^T (Z+ \Pi_0) g + \delta - D^T x g^T\ones - D^T \ones  g^T y )\}\\
\EA\right.\]
\STATE Set $k=k+1$.
\ENDWHILE
\ENSURE A doubly stochastic matrix $\Pi$.
\end{algorithmic}
\end{algorithm}

\section{Applications \& numerical experiments}\label{s:apps}

We now study the performance of the relaxations detailed above in some classical applications of seriation. Other applications  not discussed here include: social networks, sociology, cartography, ecology, operations research, psychology \citep{Liiv10}.

In most of the examples below, we will compare the performance of the spectral solution, that of the QP relaxation in~\eqref{eq:relaxedPbReg1} and the semi-supervised seriation QP in~\eqref{eq:relaxedPbRegAddCons}. In the semi-supervised experiments, we randomly sample pairwise orderings either from the true order information (if known), or from noisy ordering information. We use a simple symmetric Erd\"os-R\'enyi model for collecting these samples, so that a pair of indices $(i,j)$ is included with probability $p$, with orderings sampled independently. \citet{Erdo60} show that there is a sharp phase transition in the connectivity of the sampled graphs, with the graphs being almost surely disconnected when $p<\frac{(1-\epsilon) \log n}{n}$ and almost surely connected when $p>\frac{(1+\epsilon) \log n}{n}$ for $\epsilon>0$ and $n$ large enough. Above that threshold, i.e. when $O(n\log n)$ pairwise orders are specified, the graph is fully connected so the full variable ordering is specified {\em if the ordering information is noiseless}. Of course, when the samples include errors, some of the sampled pairwise orderings could be inconsistent, so the total order is not fully specified.

\subsection{Archeology}
We reorder the rows of the Hodson's Munsingen dataset (as provided by \cite{Hods68} and manually ordered by \cite{Kend71}), to date 59 graves from 70 recovered artifact types (under the assumption that graves from similar periods contain similar artifacts). The results are reported in Table~\ref{tab:kendall}. We use a fraction of the pairwise orders in \cite{Kend71} to solve the semi-supervised version. Note that the original data contains errors, so Kendall's ordering cannot be fully consistent. In fact, we will see that the semi-supervised relaxation actually improves on Kendall's manual ordering.

In Figure~\ref{fig:musing} the first plot on the left shows the row ordering on 59 $\times$ 70 grave by artifacts matrix given by Kendall, the middle plot is the Fiedler solution, the plot on the right is the best QP solution from 100 experiments with different $Y$ (based on the combinatorial objective in~\eqref{eq:ser-pb}). The quality of these solutions is detailed in Table~\ref{tab:kendall}.

\begin{table}[hb]
\begin{center}
\begin{tabular}{r|c|c|c|c|c}
&\textbf{\cite{Kend71}}&\textbf{Spectral}&\textbf{QP Reg} &\textbf{QP Reg + 0.1\% }&\textbf{QP Reg + 47.5\%}\\\hline
{Kendall $\tau$}&1.00&0.75&0.73$\pm$0.22&0.76$\pm$0.16&0.97$\pm$0.01\\\hline
{Spearman $\rho$}&1.00&0.90&0.88$\pm$0.19&0.91$\pm$0.16&1.00$\pm$0.00\\\hline
{Comb. Obj.}&38520 &38903&41810$\pm$13960&43457$\pm$23004& 37602$\pm$775\\\hline
{\# R-constr.}&1556&1802&2021$\pm$484&2050$\pm$747& 1545$\pm$43\\
\end{tabular}
\vskip 1ex
\caption{Performance metrics (median and stdev over 100 runs of the QP relaxation, for Kendall's~$\tau$, Spearman's $\rho$ ranking correlations (large values are good), the objective value in~\eqref{eq:ser-pb}, and the number of R-matrix monotonicity constraint violations (small values are good), comparing Kendall's original solution with that of the Fiedler vector, the seriation QP in~\eqref{eq:relaxedPbReg1} and the semi-supervised seriation QP in~\eqref{eq:relaxedPbRegAddCons} with 0.1\% and 47.5\% pairwise ordering constraints specified. Note that the semi-supervised solution actually improves on both Kendall's manual solution and on the spectral ordering.\label{tab:kendall}}
\end{center}
\end{table}

\begin{figure}[pht]
\begin{center}
\includegraphics[scale=0.8]{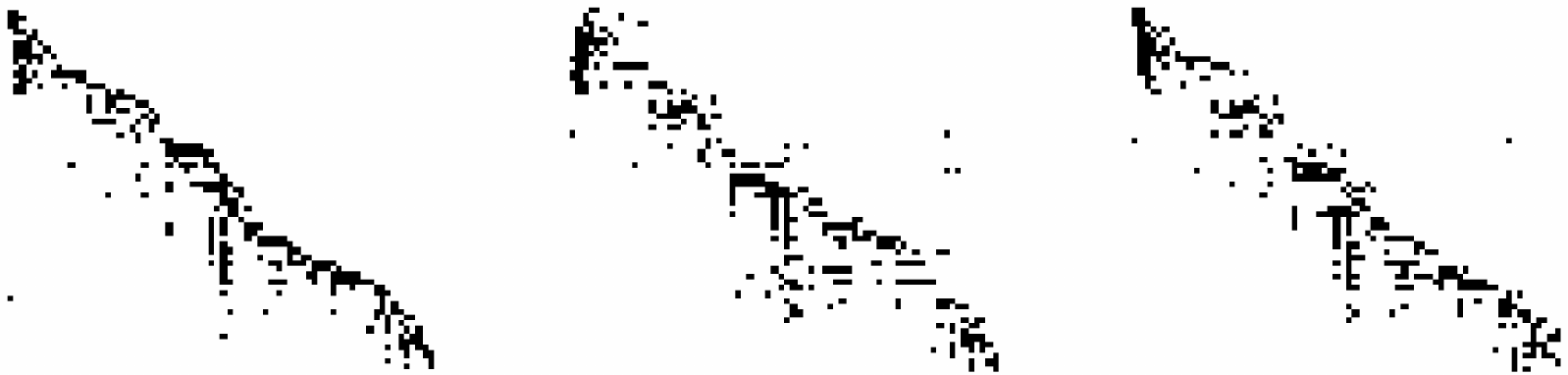}
\end{center}
\caption{The Hodson's Munsingen dataset: row ordering given by Kendall {\em (left)}, Fiedler solution {\em (center)},  best unsupervised QP solution from 100 experiments with different $Y$, based on combinatorial objective {\em (right)}.\label{fig:musing}}
\end{figure}

\subsection{Markov Chains}
Here, we observe many {\em disordered} samples from a Markov chain. The mutual information matrix of these variables must be decreasing with $|i-j|$ when ordered according to the true generating Markov chain (this is the ``data processing inequality'' in \cite[Th.\,2.8.1]{Cove12}), hence the mutual information matrix of these variables is a pre-R-matrix. We can thus recover the order of the Markov chain by solving the seriation problem on this matrix. In the following example, we try to recover the order of a Gaussian Markov chain written $X_{i+1}=b_i X_i + \epsilon_i$ with $\epsilon_i \sim N(0, \sigma_i^2)$. The results are presented in Table~\ref{tab:Markov} on 30 variables. We test performance in a noise free setting where we observe the randomly ordered model covariance,  in a noisy setting with enough samples (6000) to ensure that the spectral solution stays in a perturbative regime, and finally using much fewer samples (60) so the spectral perturbation condition fails. In Figure~\ref{fig:markov}, the first plot on the left shows the true Markov chain order, the middle plot is the Fiedler solution, the plot on the right is the best QP solution from 100 experiments with different $Y$ (based on combinatorial objective).

\begin{table}[htp]
\begin{center}
\begin{tabular}{r|c|c|c} 
&\textbf{No noise}&\textbf{Noise within spectral gap}&\textbf{Large noise}\\\hline
{True}&1.00$\pm$0.00&1.00$\pm$0.00&1.00$\pm$0.00\\\hline
{Spectral}&1.00$\pm$0.00&0.86$\pm$0.14&0.41$\pm$0.25\\\hline
{QP Reg}&0.50$\pm$0.34&0.58$\pm$0.31&0.45$\pm$0.27\\\hline
{QP + 0.2\% }&0.65$\pm$0.29&0.40$\pm$0.26&0.60$\pm$0.27\\\hline
{QP + 4.6\%}&0.71$\pm$0.08&0.70$\pm$0.07&0.68$\pm$0.08\\\hline
{QP + 54.3\%}&0.98$\pm$0.01&0.97$\pm$0.01&0.97$\pm$0.02\\
\end{tabular}
\vskip 1ex
\caption{Kendall's $\tau$ between the true Markov chain ordering, the Fiedler vector, the seriation QP in~\eqref{eq:relaxedPbReg1} and the semi-supervised seriation QP in~\eqref{eq:relaxedPbRegAddCons} with varying numbers of pairwise orders specified. We observe the (randomly ordered) model covariance matrix (no noise), the sample covariance matrix with enough samples so the error is smaller than half of the spectral gap, then a sample covariance computed using much fewer samples so the spectral perturbation condition fails.
 \label{tab:Markov}}
\end{center}\end{table}

\begin{figure}[htp]
\begin{center}
\includegraphics[scale=1]{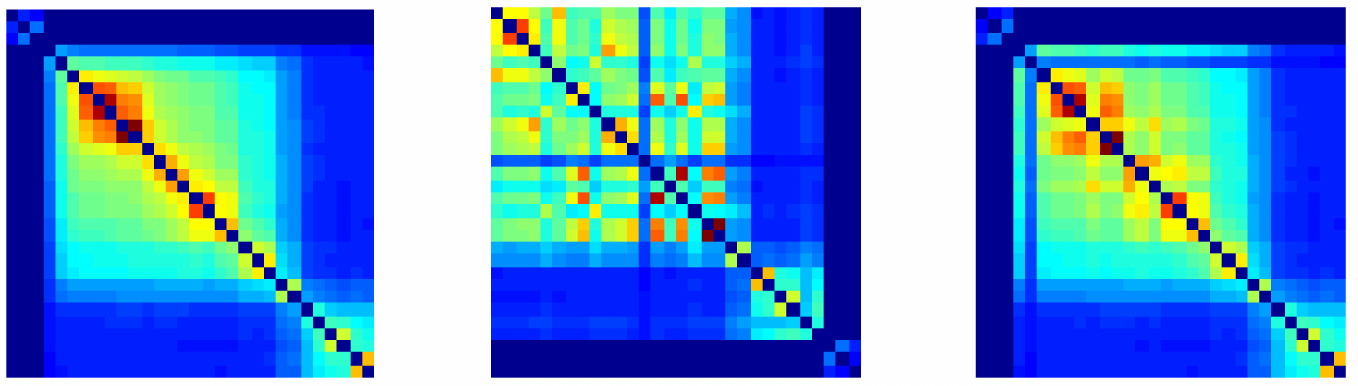}
\end{center}
\caption{Markov Chain experiments: true Markov chain order {\em (left)}, Fiedler solution {\em (center)}, best unsupervised QP solution from 100 experiments with different $Y$, based on combinatorial objective {\em (right)}.\label{fig:markov}}
\end{figure}

\subsection{Gene sequencing}
In next generation shotgun genome sequencing experiments, DNA strands are cloned about ten to a hundred times before being decomposed into very small subsequences called ``reads", each of them fifty to a few hundreds base pairs long. Current machines can only accurately sequence these small reads, which must then be reordered by ``assembly" algorithms, using the overlaps between reads. These short reads are often produced in pairs, starting from both ends of a longer sequence of known length, hence a rough estimate of the distance between these ``mate pairs'' of reads is known, giving additional structural information on the semi-supervised assembly problem.

Here, we generate artificial sequencing data by (uniformly) sampling reads from chromosome 22 of the human genome from NCBI, then store k-mer hit versus read in a binary matrix $C$ (a k-mer is a fixed sequence of~k base pairs). If the reads are ordered correctly and have identical length, this matrix is C1P, hence we solve the C1P problem on the $\{0,1\}$-matrix whose rows correspond to k-mers hits for each read, i.e.~the element $(i,j)$ of the matrix is equal to one if k-mer $j$ is present in read $i$.
The corresponding pre-R matrix obtained $CC^T$, which measures overlap between reads, is extremely sparse, as it is approximately band-diagonal with roughly constant bandwidth $b$ when reordered correctly, and computing the Fiedler vector can be done with complexity $O(bn\log n)$ w.h.p. using the Lanczos method \citep{Kucz92}, as it amounts to computing the second largest eigenvector of $\lambda_n(L) \idm - L$, where $L$ is the Laplacian of the matrix. In our experiments, computing the Fiedler vector from 250000 reads takes a few seconds using MATLAB's \texttt{eigs} on a standard desktop machine. 

In practice, besides sequencing errors (handled relatively well by the high coverage of the reads), there are often repeats in long genomes. If the repeats are longer than the k-mers, the C1P assumption is violated and the order given by the Fiedler vector is not reliable anymore. On the other hand, handling the repeats is possible using the information given by mate pairs, i.e.~reads that are known to be separated by a given number of base pairs in the original genome.  This structural knowledge can be incorporated into the relaxation~\eqref{eq:relaxedPbRegAddCons}. While our algorithm for solving~\eqref{eq:relaxedPbRegAddCons} only scales up to a few thousands base pairs on a regular desktop, it can be used to solve the sequencing problem hierarchically, i.e.~to refine the spectral solution. 

In Figure~\ref{fig:dna1}, we show the result of spectral ordering on simulated reads from human chromosome 22. The full R matrix formed by squaring the reads $\times$ kmers matrix is too large to be plotted in MATLAB and we zoom in on two diagonal block submatrices. In the first submatrix, the reordering is good and the matrix has very low bandwidth, the corresponding gene segment (called contig) is well reconstructed. In the second the reordering is less reliable, and the bandwidth is larger, so the reconstructed gene segment contains errors. 

\begin{figure}[phbt]
\begin{center}
\includegraphics[scale=0.45]{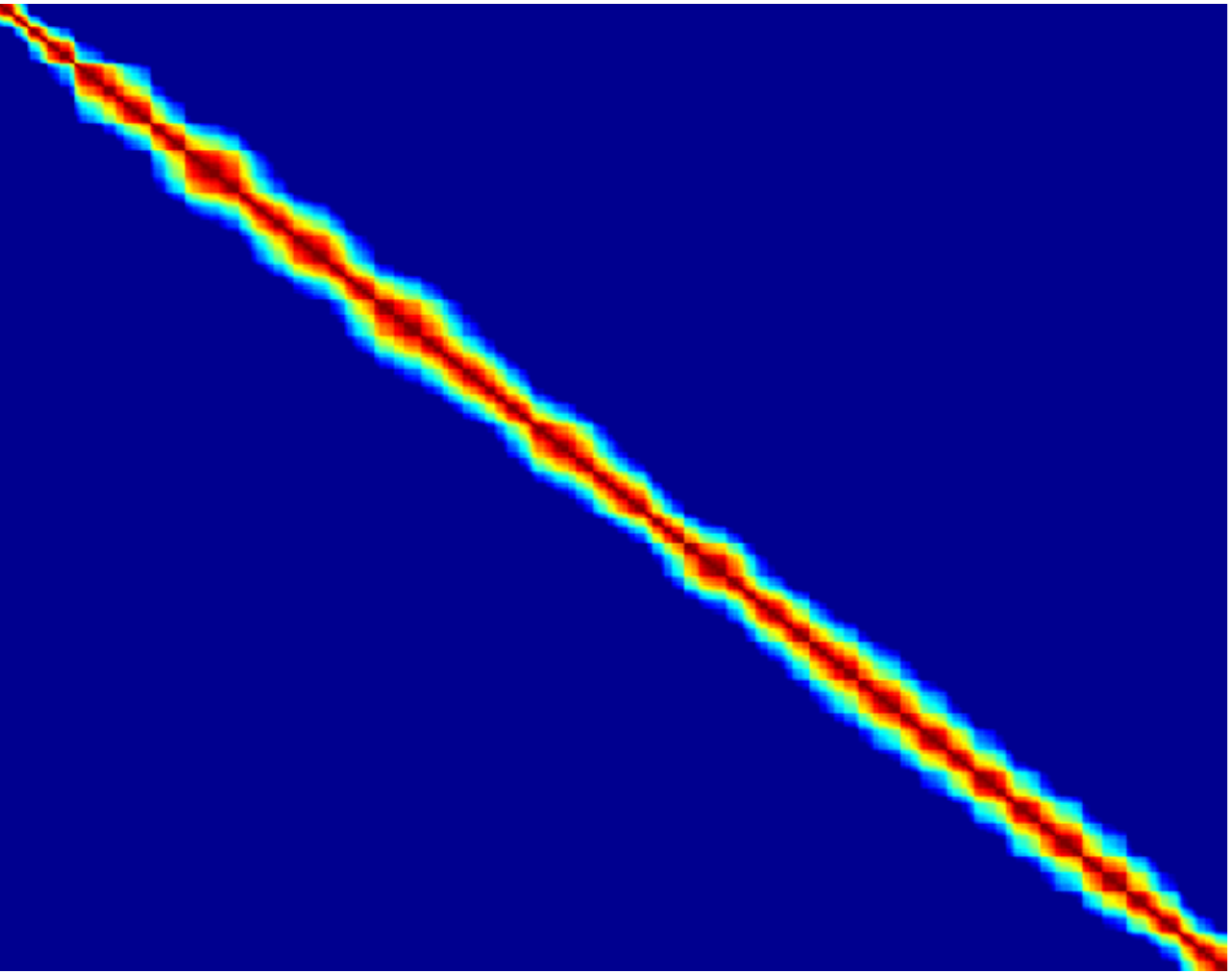}
\qquad \qquad
\includegraphics[scale=0.45]{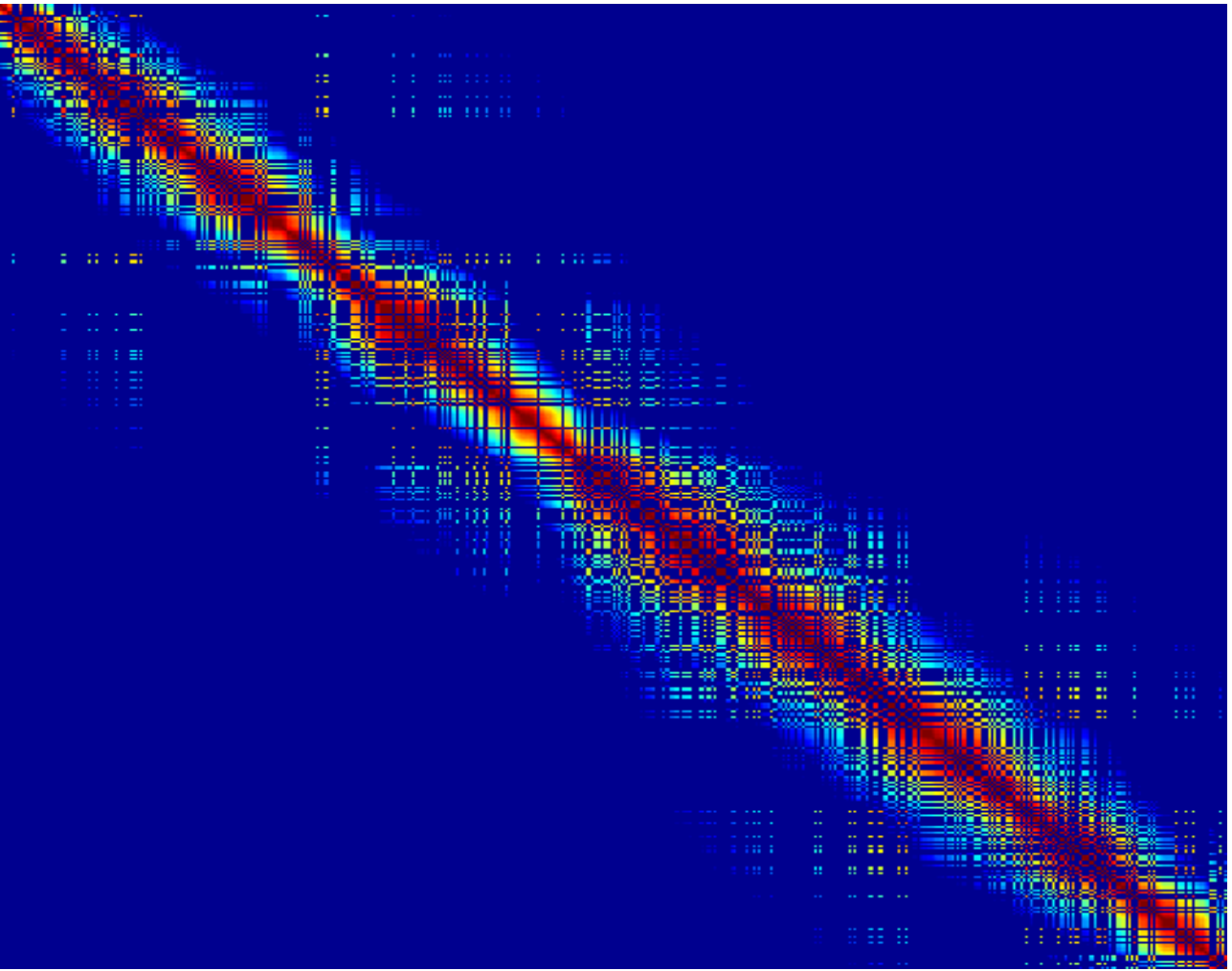}
\end{center}
\caption{We plot the $reads \times reads$ matrix measuring the number of common k-mers between read pairs, reordered according to the spectral ordering on two submatrices. \label{fig:dna1}}
\end{figure}

In Figure~\ref{fig:dna2}, we show recovered read position versus true read position for the Fiedler vector and the Fiedler vector followed by semi-supervised seriation, where the QP relaxation is applied to groups of reads (contigs) assembled by the spectral solution, on the 250\,000 reads generated in our experiments. The spectral solution orders most of these reads correctly, which means that the relaxation is solved on a matrix of dimension about $100$. We see that the number of misplaced reads significantly decreases in the semi-supervised seriation solution. Looking at the correlation between the true positions and the retrieved positions of the reads, both Kendall $\tau$ and Spearman $\rho$ are equal to one for Fiedler+QP ordering while they are equal to respectively 0.87 and 0.96 for Fiedler ordering alone. A more complete description of the assembly algorithm is given in the appendix.

\begin{figure}[phbt]
\begin{center}
\includegraphics[scale=0.4]{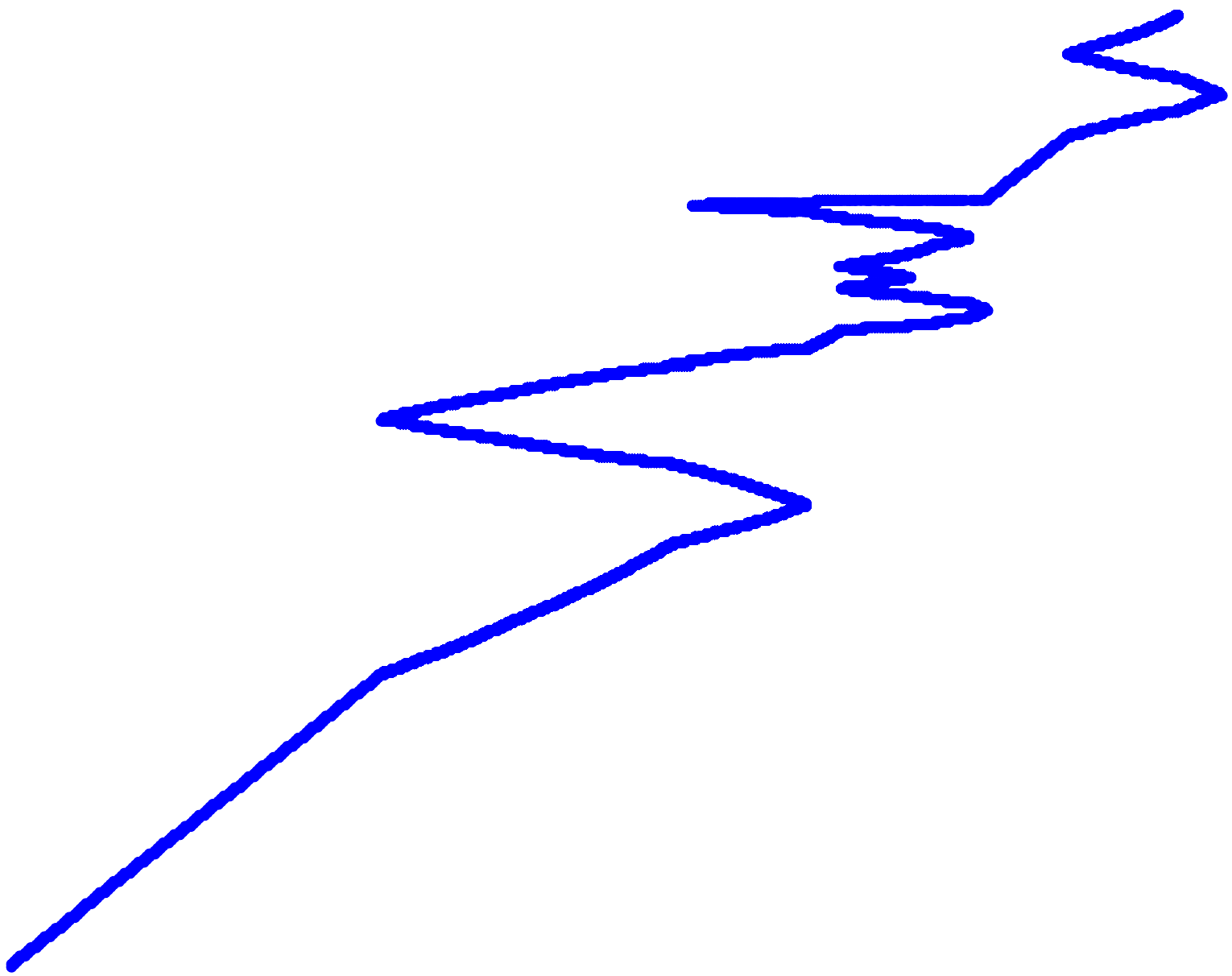}
\qquad
\includegraphics[scale=0.4]{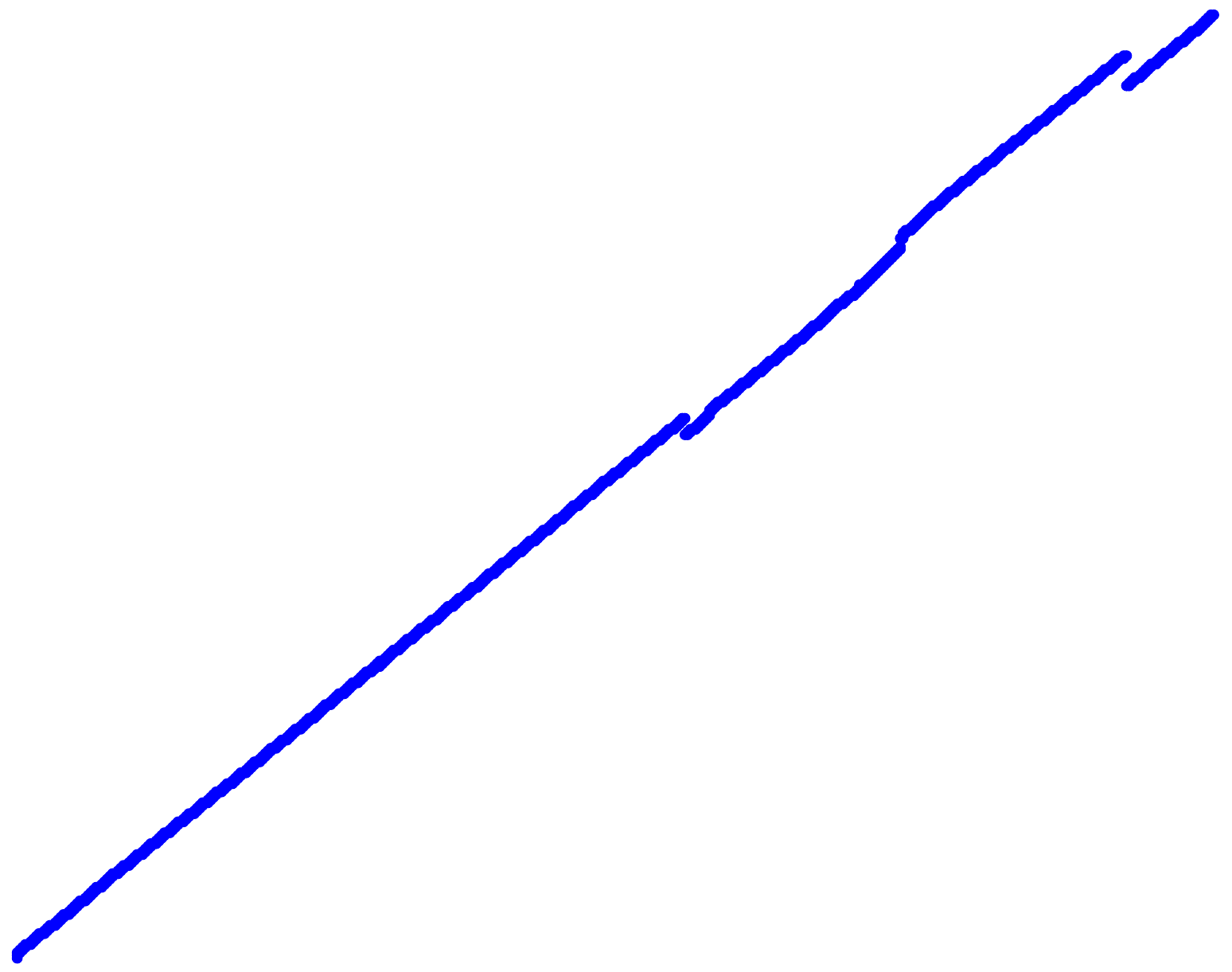}
\end{center}
\caption{We plot the Fiedler and Fiedler+QP read orderings versus true ordering. The semi-supervised solution contains much fewer misplaced reads. \label{fig:dna2}}
\end{figure}

\subsection{Generating $Y$}
We conclude by testing the impact of $Y$ on the performance of the QP relaxation in~\eqref{eq:relaxedPbRegIdeal} on a simple ranking example. In Figure~\ref{fig:Y}, we generate several matrices $Y \in \reals^{n\times p}$ as in \S\ref{sss:regularized} and  compare the quality of the solutions (permutations issued from the procedure described in~\ref{sec:sampling-proc}) obtained for various values of the number of columns $p$. On the left, we plot the histogram of the values of $g^T\Pi^TL_A\Pi g$ obtained for 100 solutions with random matrices $Y$ where $p=1$ (i.e. rank one). On the right, we compare these results with the average value of $g^T\Pi^TL_A\Pi g$ for solutions obtained with random matrices $Y$ with $p$ varying from 1 to~$5n$ (sample of 50 random matrices $Y$ for each value of $p$). The red horizontal line, represents the best solution obtained for $p=1$ over all experiments. By raising the value of $\lambda_1(YY^T)$, larger values of $p$ allow for higher values of $\mu$ in Proposition~\ref{prop:opt}, which seems to have a positive effect on performance until a point where $p$ is much larger than $n$ and the improvement becomes insignificant. We do not have an intuitive explanation for this behavior at this point.


\begin{figure}[phbt]
\begin{center}
\psfrag{objValue}[t][b]{Objective value}
\includegraphics[scale=0.49]{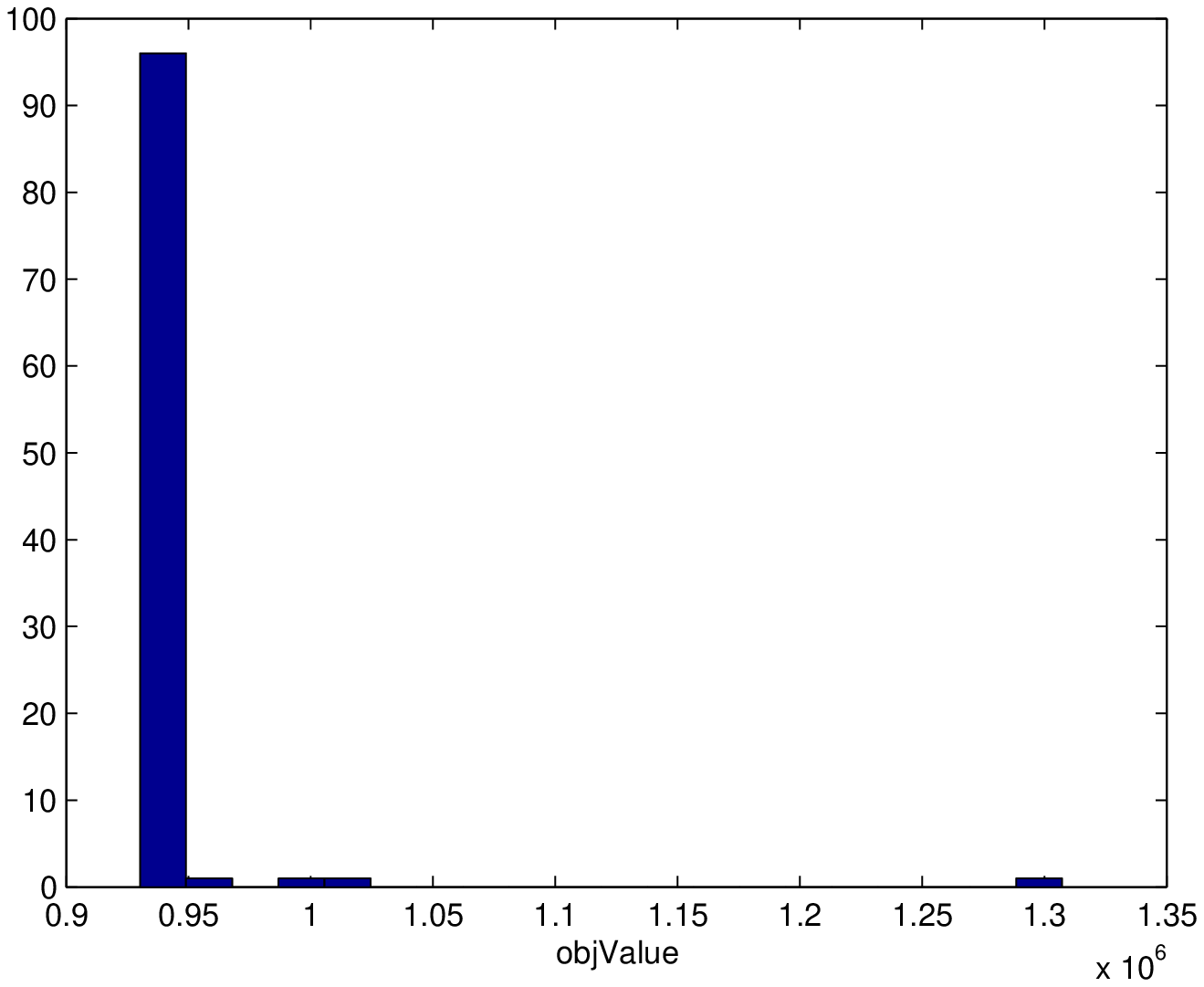}
\qquad
\psfrag{objValue}[b][t]{Objective value}
\psfrag{povern}[t][b]{$p/n$}
\includegraphics[scale=0.49]{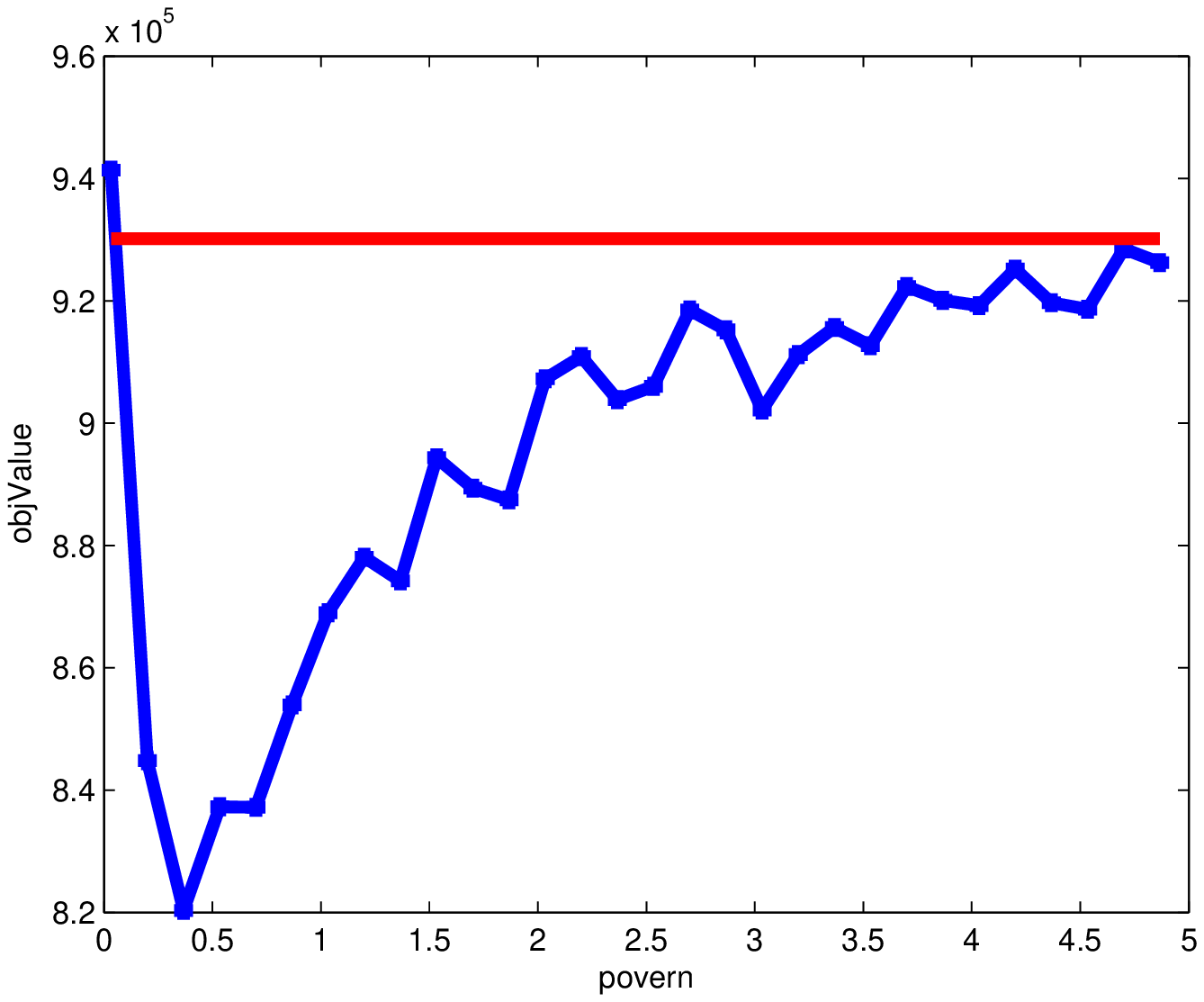}
\end{center}
\caption{{\em Left:} we plot the histogram of the values of $g^T\Pi^TL_A\Pi g$ obtained for 50 solutions with random matrices $Y$ where $p=1$ (i.e. rank one). {\em Right:} we compare these results with the average value of $g^T\Pi^TL_A\Pi g$ for solutions obtained with random matrices $Y$ with $p$ varying from 1 to $5n$. The red horizontal line, represents the best solution obtained for $p=1$ over all experiments. \label{fig:Y}}
\end{figure}

\section*{Acknowledgments}
AA, FF and RJ would like to acknowledge support from a European Research Council starting grant (project SIPA) and a gift from Google. FB would like to acknowledge support from a European Research Council starting grant (project SIERRA). A subset of these results was presented at NIPS 2013.

\section{Appendix}

In this appendix, we first briefly detail two semidefinite programming based relaxations for the 2-SUM problem, one derived from results in \citep{Nest98a,dAsp10}, the other adapted from work on the Minimum Linear Arrangement (MLA) problem in \citep{Even00,Feig00,Blum00} among many others. While their complexity is effectively too high to make them practical seriation algorithms, these relaxations come with explicit approximation bounds which aren't yet available for the QP relaxation in Section~\ref{ss:qp}. These SDP relaxations illustrate a clear tradeoff between approximation ratios and computational complexity: low complexity but unknown approx. ratio for the QP relaxation in~\eqref{eq:relaxedPb1}, high complexity and $\sqrt{n}$ approximation ratio for the first semidefinite relaxation, very high complexity but excellent $\sqrt{\log n}$ approximation ratio for the second SDP relaxation. The question of how to derive convex relaxations with nearly dimension independent approximation ratios (e.g. $O(\sqrt{\log n})$) and good computational complexity remains open at this point. 

We then describe in detail the data sets and procedures used in the DNA sequencing experiments of Section~\ref{s:apps}.

\subsection{SDP relaxations \& doubly stochastic matrices}\label{ss:approx-qp}
Using randomization techniques derived from \citep{Nest98a,dAsp10}, we can produce approximation bounds for a relaxation of the nonconvex QP representation of~\eqref{eq:ser-pb} derived in~\eqref{eq:relaxedPbReg1}, namely
\[\BA{ll}\label{eq:qp-sdp}
\mbox{minimize} &  \Tr(Y^T\Pi^T L_A\Pi Y) - {\mu}\|P\Pi\|_F^2\\
\mbox{subject to} & \Pi \ones=\ones,\, \Pi^T\ones=\ones, \, \Pi \geq 0,
\EA\]
which is a (possibly non convex) quadratic program in the matrix variable $\Pi\in\reals ^{n\times n}$, where $P=\idm - \frac{1}{n}\mathbf{1}\mathbf{1}^T$. We now set the penalty $\mu>0$ sufficiently high to ensure that the objective is concave and the constraint $\|\Pi\|=\sqrt{n}$ is saturated. From Proposition~\eqref{prop:opt} above, this means $\mu>\|L_A\|_2 \|Y\|_2^2$. The solution of this concave minimization problem over the convex set of doubly stochastic matrices will then be at an extreme point, i.e. a permutation matrix. We first rewrite the above QP as a more classical maximization problem over vectors
\[\BA{ll}
\mbox{maximize} &  \left(\vect{\Pi}^T{\mathcal A}\vect{\Pi}\right)^{1/2}\\
\mbox{subject to} & (\ones^T\otimes \idm)\vect{\Pi}=\ones,\, (\idm \otimes \ones^T)\vect{\Pi}=\ones, \, \Pi \geq 0.
\EA\]
We use a square root in the objective here to maintain the same homogeneity properties as in the linear arrangement problems that follow. Because the objective is constructed from a Laplacian matrix, we have $\ones^T{\mathcal A}\ones=0$ so the objective is invariant by a shift in the variables. We now show that the equality constraints can be relaxed without loss of generality. We first recall a simple scaling algorithm due to \citep{Sink64} which shows how to normalize to one the row and column sums of a strictly positive matrix. Other algorithms based on geometric programming with explicit complexity bounds can be found in e.g. \citep{Nemi99a}.

\begin{algorithm}[ht]
\caption{Matrix scaling (Sinkhorn).}
\label{alg:sinkhorn}
\begin{algorithmic} [1]
\REQUIRE A matrix $\Pi\in\reals^{m \times n}$
\FOR{$k=1$ to $N-1$}
\STATE Scale row sums to one: $\Pi_{k+1/2}=\diag(\Pi_k \ones)^{-1}\Pi_k$
\STATE Scale column sums to one: $\Pi_{k+1}=\Pi_{k+1/2}\diag(\ones^T \Pi_{k+1/2} )^{-1}$
\ENDFOR
\ENSURE A scaled matrix $\Pi_N$.
\end{algorithmic}
\end{algorithm}

The next lemma shows that the only matrices satisfying both $\|\Pi\|_F=\sqrt{n}$ and $\Pi \ones\leq\ones$, $\Pi^T\ones\leq\ones,$ with $\Pi \geq 0$ are doubly stochastic.

\begin{lemma} Let $\Pi \in \reals^{n \times n}$, if $\|\Pi\|_F=\sqrt{n}$ and $\Pi \ones\leq\ones$, $\Pi^T\ones\leq\ones,$ with $\Pi \geq 0$, then $\Pi$ is doubly stochastic.
\end{lemma}
\begin{proof}
Suppose $\Pi \ones \leq \ones,\, \Pi^T\ones \leq\ones, \, \Pi > 0$, each iteration of Algorithm~\ref{alg:sinkhorn} multiplies $\vect{\Pi}$ by a diagonal matrix $D$ with diagonal coefficients greater than one, with at least one coefficient strictly greater than one if $\Pi$ is not doubly stochastic, hence $\|\Pi\|_F$ is strictly increasing if $\Pi$ is not doubly stochastic. This means that the only maximizers of $\|\Pi\|_F$ over the feasible set of~\eqref{eq:qp-sdp} are doubly stochastic matrices.
\end{proof}

We let $z=\vect{\Pi}$, the above lemma means that problem~\eqref{eq:qp-sdp} is equivalent to the following QP
\BEQ\label{eq:qp-ds}\tag{QP}
\BA{ll}
\mbox{maximize} &  \|{\mathcal A}^{1/2}z\|_2\\
\mbox{subject to} & (\ones^T\otimes \idm)z \leq \ones,\, (\idm \otimes \ones^T)z \leq \ones,\\ 
& z \geq 0,
\EA\EEQ
in the variable $z\in\reals^{n^2}$. Furthermore, since permutation matrices are binary matrices, we can impose the redundant constraints that $z_i\in\{0,1\}$ or equivalently $z_i^2=z_i$ at the optimum. Lifting the quadratic objective and constraints as in \citep{Shor87,Lova91} yields the following relaxation
\BEQ\label{eq:sdp-ds}\tag{SDP1}
\BA{ll}
\mbox{maximize} &  \Tr({\mathcal A}Z)\\
\mbox{subject to} & (\ones^T\otimes \idm)z \leq \ones,\, (\idm \otimes \ones^T)z \leq \ones,\\ 
& Z_{ii}=z_i,\,Z_{ij}\geq 0, \quad i,j=1,\ldots,n,\\
& 
\left(\BA{cc}
Z & z\\
z^T & 1
\EA\right)
\succeq 0,
\EA\EEQ
which is a semidefinite program in the matrix variable $Z\in\symm_{n^2}$ and the vector $z\in\reals^{n^2}$. By adapting a randomization argument used in the MaxCut relaxation bound in \citep{Nest98a} and adapted to the $k$-dense-subgraph problem in \citep{dAsp10}, we can show the following $O(\sqrt{n})$ approximation bound on the quality of this relaxation.

\begin{proposition}\label{prop:qp-ds}
Let OPT be the optimal value of problem~\eqref{eq:qp-ds} and SDP1 be that of~\eqref{eq:sdp-ds}, then
\[
0 \leq \frac{\Tr(\mathcal{A}G)}{4 n}+\frac{SDP1}{2\pi n} \leq OPT^2 \leq SDP1.
\]
with $G_{ij}=\sqrt{Z_{ii}Z_{jj}}$, $i=1,\ldots,n$ and $\Tr(\mathcal{A}G)\leq 0$.
\end{proposition}
\begin{proof}
The fact that ${\mathcal A}\succeq 0$ by construction shows $0\leq OPT^2\leq SDP1$. Let $\xi\sim\mathcal{N}(0,Z)$, and define
\[
y_i=
\left\{\BA{cl}
\sqrt{z_i} & \mbox{if $\xi_i \geq 0$}\\ 
0 & \mbox{otherwise.}
\EA\right.
\]
We write $C=\diag(Z)^{-1/2}Z\diag(Z)^{-1/2}$ the correlation matrix associated with $Z$ (under the convention that $C_{ij}=0$ whenever $Z_{ii}Z_{jj}=0$). A classical result from \citep{Shep00} (see also \citep[p.95]{John72}) shows
\[
\Expect[y_i y_j]=\sqrt{z_iz_j}\left(\frac{1}{4}+\frac{1}{2\pi} \arcsin(C)\right), \quad i=1,\ldots,n,
\]
and $\mathcal{A}\succeq 0$ together with $\arcsin(C)\succeq C$ (with the $\arcsin(\cdot)$ taken elementwise) and $z_i=Z_{ii}$ means that, writing $G_{ij}=\sqrt{z_iz_j}=\sqrt{Z_{ii}Z_{jj}}$, we get 
\BEAS
\Expect[y^T\mathcal{A}y] &= &\Expect[\Tr(\mathcal{A}yy^T)]\\
& = & \Tr\left(\mathcal{A}\left(G\circ \left(\frac{1}{4}\ones\ones^T+\frac{1}{2\pi} \arcsin(C)\right)\right)\right)\\
& \leq & \Tr\left(\mathcal{A}\left(\frac{1}{4}G+\frac{1}{2\pi} Z\right)\right)\\
& = & \frac{1}{4}\Tr(\mathcal{A}G)+\frac{1}{2\pi} SDP1,
\EEAS
because Schur's theorem shows that $A\circ B \succeq 0$ when $A,B \succeq 0$. It remains to notice that, because $(\ones^T\otimes \idm)z \leq \ones,$ and $(\idm \otimes \ones^T)z \leq \ones$, with $z\geq 0$, then 
\[
(\ones^T\otimes \idm)\sqrt{z} \leq \sqrt{n} \ones, \quad \mbox{and} \quad (\idm \otimes \ones^T)\sqrt{z} \leq \sqrt{n} \ones,
\]
so all the points $y$ generated using this procedure are feasible for~\eqref{eq:qp-ds} if we scale them by a factor $\sqrt{n}$. 
\end{proof}

While the $O(\sqrt{n})$ bound grows relatively fast with problem dimension, remember that the problem has $n^2$ variables because it is written on permutation {\em matrices}. In what follows, we will see that better theoretical approximation bounds can be found if we write the seriation problem directly over permutation {\em vectors}, which is of course a much more restrictive formulation.

\subsection{SDP relaxations \& minimum linear arrangement}\label{ss:mla}
Several other semidefinite relaxations have been derived for the 2-SUM problem and the directly related 1-SUM, or {\em minimum linear arrangement} (MLA) problem. While these relaxations have very high computational complexity, to the point of being impractical, they come with excellent approximation bounds. We briefly recall these results in what follows. The 2-SUM minimization problem~\eqref{eq:2sum-pb} is written (after taking square roots)
\BEQ\label{eq:sqrt-2sum}
\BA{ll}
\mbox{minimize} & \left(\sum_{i,j=1}^n A_{ij}(\pi(i)-\pi(j))^2\right)^\frac{1}{2}\\
\mbox{subject to} & \pi \in \mathcal{P}.
\EA\EEQ
in the variable $\pi\in \mathcal{P}$ which is a permutation of the vector $(1,\ldots,n)^T$. \citet{Even00,Feig00,Blum00} form the following semidefinite relaxation 
\BEQ\label{eq:SDP2}\tag{SDP2}
\BA{ll}
\mbox{minimize} & \sum_{i,j=1}^n A_{ij}X_{ij}\\
\mbox{subject to} & \frac{1}{|S|} \sum_{j\in S} (X_{ii}-2X_{ij}+X_{jj}) \geq \frac{1}{6}(|S|/2+1)(|S|+1),\quad \mbox{for all }S\subset[1,n],~i=1,\ldots,n\\
&\frac{1}{|S|} \sum_{k\in S} \Delta^2(i,j,k)\geq \epsilon (X_{ii}-2X_{ij}+X_{jj}) |S|^2,\quad \mbox{for all }S\subset[1,n],~i,j=1,\ldots,n\\
& X\succeq 0,\,X_{ij}\geq 0, \quad i,j=1,\ldots,n
\EA\EEQ
in the variable $X\in\symm_n$, where $\epsilon >0$ and $\Delta(i,j,k)$ is given by the determinant
\[
\Delta(i,j,k)=\left|
\BA{cc}
X_{jj}-2X_{ij}+X_{ii} & X_{jk}-X_{ij}-X_{jk}+X_{ii}\\
X_{jk}-X_{ij}-X_{jk}+X_{ii} & X_{kk}-2X_{ik}+X_{ii}
\EA\right|.
\]
\citep[Th.\,2]{Blum00} shows that if OPT is the optimal value of the 2-SUM problem~\eqref{eq:sqrt-2sum} and SDP2 the optimal value of the relaxation in~\eqref{eq:SDP2}, then 
\[
SDP2 ~(\log n)^{-1/2} \leq OPT \leq SDP2~(\log n)^{3/2}.
\]
While problem~\eqref{eq:SDP2} has an exponential number of constraints, efficient linear separation oracles can be constructed for the last two spreading constraints, hence the problem can be solved in polynomial time \citep{Grot88}. 

Tighter bounds can be obtained by exploiting approximation results on the minimum linear arrangement problem, noting that, after taking the square root of the objective, the 2-SUM problem is equivalent to
\BEQ\label{eq:minmax-mla}
\min_{\pi \in \mathcal{P}}~\max_{\{\|V\|_F\leq 1,\, V \geq 0\}}~ \sum_{i,j=1}^n V_{ij}A_{ij}^{1/2}|\pi(i)-\pi(j)|
\EEQ
in the variables $\pi\in \mathcal{P}$ and $V\in\reals^{n \times n}$ (note that this is true for the support function of any set contained in the nonnegative orthant). Using results in \citep{Rao05,Feig07,Char10}, the minimum linear arrangement problem, written
\BEQ\label{eq:mla}\tag{MLA}
\min_{\pi \in \mathcal{P}}~ \sum_{i,j=1}^n W_{ij}|\pi(i)-\pi(j)|
\EEQ
over the variable $\pi\in \mathcal{P}$, with nonnegative weights $W\in\reals^{n \times n}$, can be relaxed as
\BEQ\label{eq:SDP3}\tag{SDP3}
\BA{ll}
\mbox{minimize} & \sum_{i,j=1}^n W_{ij}(X_{ii}-2X_{ij}+X_{jj})\\
\mbox{subject to} & \frac{1}{|S|} \sum_{j\in S} (X_{ii}-2X_{ij}+X_{jj}) \geq \frac{|S|^2}{5},\quad \mbox{for all }S\subset[1,n],~i=1,\ldots,n\\
 & (X_{ii}-2X_{ij}+X_{jj}) \leq (X_{ii}-2X_{ik}+X_{kk}) + (X_{kk}-2X_{kj}+X_{jj}),\quad i,j,k=1,\ldots,n\\
& (X_{ii}-2X_{ij}+X_{jj}) \geq 1,\quad i,j=1,\ldots,n\\
& X \succeq 0,
\EA\EEQ
in the variable $X\in\symm_n$. The constraints above ensure that $d_{ij}=(X_{ii}-2X_{ij}+X_{jj})$ is a squared Euclidean metric (hence a metric of negative type). If MLA is the optimal value of the minimum linear arrangement problem~\eqref{eq:mla} and SDP3 the optimum of the relaxation in~\eqref{eq:SDP3}, \citep[Th.\,2.1]{Feig07} and \citep{Char10} show that
\[
SDP3 \leq MLA \leq SDP3~O(\sqrt{\log n}\, \log\log n),
\]
which immediately yields a convex relaxation with $O(\sqrt{\log n}\, \log\log n)$ approximation ratio for the minmax formulation of the 2-SUM problem in~\eqref{eq:minmax-mla}.

\subsection{Procedure for gene sequencing}
We first order all the reads using the spectral algorithm. Then, in order to handle repeats in the DNA sequence,  we adopt a divide and conquer approach and reorder smaller groups of reads partitioned using the spectral order. Finally we use the information given by mate pairs to reorder the resulting clusters of reads, using the QP relaxation. Outside of spectral computations which take less than a minute in our experiments, most computations can be naively parallelized. The details of the procedure are given below.

\begin{itemize}
\item Extract uniformly reads of length a few hundreds bp (base pairs) from DNA sequence. In our experiments, we artificially extract reads of length 200 bp from the true sequence of a million bp of the human chromosome 22. We perform a high coverage (each bp is contained in approx. 50 reads) uniform sampling. To replicate the setting of real sequencing data, we extract pairs of reads, with a distance of 5000 bp between each ``mate" pairs.

\item Extract all possible k-mers from reads, i.e.~for each read, record all subsequence of size k. We use k=100 in our experiments. The size of k-mers may be tuned to deal with noise in sequencing data (use small k) or repeats (use large k).
\item Solve the C1P problem on the $\{0,1\}$-matrix whose rows correspond to k-mers hits for each read, i.e.~the element $(i,j)$ of the matrix is equal to one if k-mer $j$ is included in read $i$.
Note that solving this C1P problem corresponds to reordering the similarity matrix between reads whose element $(r,s)$ is the number of shared k-mers between reads $r$ and $s$. In the presence of noise in sequencing data, this similarity matrix can be made more robust by recomputing for instance an edit distance between reads sharing k-mers.
Moreover, if there are no repeated k-mers in the original sequence, i.e.~a k-mer appears in two reads only if they overlap in the original sequence, then the C1P problem is solved exactly by the spectral relaxation and the original DNA sequence is retrieved by concatenating the overlapping reordered reads. Unfortunately, for large sequences, repeats are frequent and the spectral solution ``mixes" together different areas of the original sequence. We deal with repeats in what follows.
\item We extract contigs from the reordered reads: extract with high coverage (e.g.~10) sequences of a few thousands reads from the reordered sequence of reads (250\,000 reads in our experiments). Although there were repeats in the whole sequence, a good proportion of the contigs do not contain reads with repeats. By reordering each contig (using the spectral relaxation) and looking at the corresponding similarity (R-like) matrix, we can discriminate between ``good" contigs (with no repeats and therefore a perfectly reordered similarity matrix which is an R-matrix) and ``bad" contigs (with repeats and a badly reordered similarity matrix).
\item Reorder the ``good" contigs from the previous step using the spectral relaxation and agglomerate overlapping contigs. The aggregation can be done using again the spectral algorithm on the sub matrix of the original similarity matrix corresponding to the two clusters of reads.
Now there should be only a few (long) contigs left (usually less than a few hundreds in our experiments).
\item Use the mate pairs to refine the order of the contigs with the QP relaxation to solve the semi-supervised seriation problem. Gaps are filled by incorporating the reads from the ``bad" contigs (contigs with repeats).
\end{itemize}

Overall, the spectral preprocessing usually shrinks the ordering problem down to dimension $n\sim 100$, which is then solvable using the convex relaxations detailed in Section~\ref{s:relax}.



{\bibliographystyle{agsm}
\bibliography{/Users/aspremon/Dropbox/Research/Biblio/MainPerso.bib}}

\end{document}